\documentclass[12pt]{amsart}
\usepackage[colorlinks,citecolor=blue,urlcolor=black,linkcolor=black]{hyperref}
\usepackage{amssymb}

\newtheorem{thm}{Theorem}[section]

\newtheorem*{THMA}{Theorem A}
\newtheorem*{THMB}{Theorem B}
\newtheorem*{THMD}{Theorem D}
\newtheorem*{THMC}{Theorem C}
\newtheorem{lemma}[thm]{Lemma}
\newtheorem{cor}[thm]{Corollary}
\newtheorem{claim}{Claim}[thm]
\newtheorem{prop}[thm]{Proposition}
\newtheorem{q}{Question}

\theoremstyle{definition}
\newtheorem{defn}[thm]{Definition}
\newtheorem{remark}[thm]{Remark}

\theoremstyle{remark}
\newtheorem*{note}{Note}

\def\s{\subseteq}
\def\bks{\setminus}
\def\br{\blacktriangleright}
\def\npg{}

\def\sd{\framebox[3.6mm][l]{$\diamondsuit$}\hspace{0.5mm}{}}
\def\sc{\framebox[3.6mm][l]{$\clubsuit$}\hspace{0.5mm}{}}

\def\fsc{\framebox[2.9mm][l]{$\clubsuit$}\hspace{0.5mm}{}}

\DeclareMathOperator{\reg}{Reg}
\DeclareMathOperator{\cf}{cf}
\DeclareMathOperator{\refl}{Refl}
\DeclareMathOperator{\im}{Im}
\DeclareMathOperator{\dom}{dom}
\DeclareMathOperator{\otp}{otp}
\DeclareMathOperator{\acc}{acc}
\DeclareMathOperator{\ACC}{Acc}
\DeclareMathOperator{\nacc}{nacc}
\DeclareMathOperator{\NACC}{Nacc}
\DeclareMathOperator{\zfc}{\hbox{\textsf{ZFC}}}
\DeclareMathOperator{\gch}{\hbox{\textsf{GCH}}}
\DeclareMathOperator{\ch}{CH}

\DeclareMathOperator{\drop}{Drop}

\DeclareMathOperator{\T}{\hbox{\textbf{T}}}
\DeclareMathOperator{\h}{ht}

\begin{document}
\author{Assaf Rinot}
\address{The Center for Advanced Studies in Mathematics,\\
Ben-Gurion University of the Negev,\\
P.O.B. 653, Be'er Sheva 84105, Israel.}
\email{paper11@rinot.com}
\urladdr{http://www.assafrinot.com}

\title[The Ostaszewski square]{The Ostaszewski square, and homogenous Souslin trees}
\begin{abstract} Assume $\gch$ and let $\lambda$ denote an uncountable cardinal.
We prove that if $\square_\lambda$ holds,
then this may be  witnessed by a coherent sequence $\left\langle C_\alpha \mid \alpha<\lambda^+\right\rangle$
with the following remarkable guessing property:

For every sequence
$\langle A_i\mid i<\lambda\rangle$
of unbounded subsets of $\lambda^+$,
and every limit $\theta<\lambda$,
there exists some $\alpha<\lambda^+$ such that
$\otp(C_\alpha)=\theta$, and the $(i+1)_{th}$-element of $C_\alpha$ is a member of $A_i$, for all $i<\theta$.

As an application, we construct an homogenous $\lambda^+$-Souslin tree
from $\textsf{GCH}+\square_\lambda$, for every singular cardinal $\lambda$.

In addition, as a by-product, a theorem of Farah and Veli\v{c}kovi\'{c}, and a theorem of Abraham, Shelah and Solovay
are generalized to cover the case of successors of regulars.

\end{abstract}
\maketitle

\setcounter{tocdepth}{1}
\tableofcontents

\newpage\section{Introduction}
\subsection*{Background and results}
For a cardinal $\lambda$, Jensen's square principle, $\square_\lambda$, asserts the existence of a sequence $\langle C_\alpha\mid \alpha<\lambda^+\rangle$
such that for every limit ordinal $\alpha<\lambda^+$:
\begin{enumerate}
\item $C_\alpha$ is a club in $\alpha$ of order-type $\le\lambda$;
\item if $\beta\in\acc(C_\alpha)$, then $C_\beta=C_\alpha\cap\beta$.\footnote{Here, $\acc(C_\alpha):=\{\beta\in C_\alpha\mid \sup(C_\alpha\cap\beta)=\beta\}$
stands for the set of accumulation points of $C_\alpha$. Similarly, we define $\nacc(C_\alpha):=C_\alpha\bks\acc(C_\alpha)$.}
\end{enumerate}

$\square_\omega$ is a consequence of \textsf{ZFC}, while $\square_\lambda$ for an uncountable $\lambda$,
is a principle independent \textsf{ZFC}. Jensen \cite{jensen} proved that if $V=L$,
then $\square_\lambda$ holds for every (uncountable) cardinal $\lambda$, and utilized this fact in
proving that in G\"odel's constructible universe $L$, for every uncountable cardinal $\lambda$, there exists a $\lambda^+$-Souslin.

One of the basic observations concerning square is that if $\overrightarrow C=\langle C_\alpha\mid \alpha<\lambda^+\rangle$
is a $\square_\lambda$-sequence, then up to some trivial modifications, so does $\ACC(\overrightarrow C):=\langle \acc(C_\alpha)\mid \alpha<\lambda^+\rangle$.
This suggests that the sequence $\NACC(\overrightarrow C):=\langle \nacc(C_\alpha)\mid \alpha<\lambda^+\rangle$
is of no interest.
And indeed, at least to the best of our knowledge,  the current literature omits the study of this object.

In this paper, we shall show that $\NACC(\overrightarrow C)$ can be as wild as one can imagine,
and demonstrate that the move from $\overrightarrow C$ to $\ACC(\overrightarrow C)$ may lead to the loss of a treasure,
in the sense that $\NACC(\ACC(\overrightarrow C))$ may be considerably poorer than  $\NACC(\overrightarrow C)$.
For this, we introduce a syntactical  strengthening of $\square_\lambda$ which we denote by $\sc_\lambda$,
and prove that the latter follows from the former,  provided that certain fragments of the $\gch$ holds.
In particular, this yields that $L\models ``\sc_\lambda\text{ is valid for every uncountable cardinal }\lambda\text{''}$.

Let us commence with defining the following weak variation of $\sc_\lambda$:

\begin{defn} For an infinite cardinal $\lambda$, and a stationary subset $S\s\lambda^+$,
$\sc_\lambda(S)$ asserts the existence of a $\square_\lambda$-sequence $\langle C_\alpha\mid\alpha<\lambda^+\rangle$ and
a subset $S'\s S$ such that:
\begin{itemize}
\item[(3)] for every club $D\s\lambda^+$, and every cofinal $A\s\lambda^+$,
there exists some $\alpha\in S'$ such that $\acc(C_\alpha)\s D$ and $\nacc(C_\alpha)\s A$;
\item[(4)] $S'\cap\acc(C_\alpha)=\emptyset$ for all $\alpha<\lambda^+$.
\end{itemize}
\end{defn}
\begin{note} Clause (3) is equivalent to the assertion that $\langle \nacc(C_\alpha)\mid \alpha\in S'\rangle$ is a $\clubsuit(S')$-sequence,
hence the choice of notation.
\end{note}
In \cite{velickovicfarah}, it is proved that if $\ch_\lambda+\square_\lambda$ holds for a cardinal $\lambda$ which is \emph{singular strong limit of uncountable cofinality},
then there exists an almost-measure, non-measure algebra of size $\lambda^+$.\footnote{That is,
a complete Boolean algebra of size $\lambda^+$ which is not a measure algebra,
but any complete subalgebra of strictly smaller size is a measure algebra. The principle $\ch_\lambda$ asserts that $2^\lambda=\lambda^+$. }
A second look at their proof reveals that what is actually used, is the above sort of square. More precisely:
\begin{thm}[Farah-Veli\v{c}kovi\'{c}, implicit in \cite{velickovic}] Suppose that $\sc_\lambda(E^{\lambda^+}_\omega)$ holds
for a given cardinal $\lambda\ge\mathfrak d$.

Then there exists an almost-measure, non-measure algebra of size $\lambda^+$.
\end{thm}

Now, in this paper, it is proved:
\begin{THMA} Suppose that $\ch_\lambda$ holds for a given uncountable cardinal $\lambda$.

Then all of the following are equivalent:
\begin{itemize}
\item $\square_\lambda$;
\item $\sc_\lambda(S)$, for every stationary  $S\s E^{\lambda^+}_{\not=\cf(\lambda)}$;
\item $\sc_\lambda(S)$, for every $S\s \lambda^+$ that reflects stationarily often.
\end{itemize}
\end{THMA}
\begin{note} The models from \cite{sh186},\cite{kingsteinhorn}
witness that $\textsf{GCH}+\square_\lambda$ does \emph{not} imply $\sc_\lambda(S)$ for $S\s E^{\lambda^+}_{\cf(\lambda)}$ that does not reflect stationarily often.
\end{note}
Altogether, we get that if $\ch_\lambda+\square_\lambda$ holds for a cardinal $\lambda\ge\mathfrak d$,
then there exists an almost-measure, non-measure algebra of size $\lambda^+$.
In other words, the original hypothesis \cite{velickovicfarah} that $\lambda$ is a singular strong limit of uncountable cofinality
may be reduced to just ``$\lambda\ge\mathfrak d$''.

Let us now turn to central object of this paper.

\begin{defn}[The Ostaszewski square]\label{def13} For an infinite cardinal $\lambda$,
$\sc_\lambda$ asserts the existence of a $\square_\lambda$-sequence $\langle C_\alpha\mid\alpha<\lambda^+\rangle$ satisfying:
\begin{itemize}
\item[(3)] Suppose that $\langle A_i\mid i<\lambda\rangle$
is a sequence of unbounded subsets of $\lambda^+$.
Then for every limit $\theta<\lambda$, and every club $D\s\lambda^+$,
there exists some $\alpha<\lambda^+$ such that $\otp(C_\alpha)=\theta$, and for all $i<\theta$:
\begin{enumerate}
\item[(a)] $C_\alpha(i+1)\in A_i$;\footnote{Here,
$C_\alpha(i)$ stands for the $i_{th}$ element of $C_\alpha$, that is,
the unique $\beta\in C_\alpha$ satisfying $\otp(C_\alpha\cap\beta)=i$.}
\item[(b)] $C_\alpha(i)<\beta<C_\alpha(i+1)$ for some $\beta\in D$.
\end{enumerate}
\end{itemize}
\end{defn}
\begin{note}
To assist the reader digest clause (3), we offer the following equivalent formulation.
Suppose that $\langle A_i\mid i<\lambda\rangle$ is a sequence of unbounded subsets of $\lambda^+$.
Then for every limit $\theta<\lambda$, there exists a limit $\alpha<\lambda^+$,
such that the isomorphism  $\pi_\alpha:\otp(C_\alpha)\rightarrow\nacc(C_\alpha)$
is an element of $\prod_{i<\theta}A_i$.
\end{note}
Of course, $\sc_\lambda$ implies $\sc_\lambda(E^{\lambda^+}_\theta)$
for every regular $\theta<\lambda$, but more importantly, it provides us with a much better control on the non-accumulation points of its components.
For instance, if $\overrightarrow C$ is a $\sc_\lambda$-sequence for a singular cardinal $\lambda$,
then for every continuous cofinal function $f:\cf(\lambda)\rightarrow\reg(\lambda)$, and every club $D\s\lambda^+$,
there exists some $\alpha\in E^{\lambda^+}_{\cf(\lambda)}$ such that $C_\alpha\s D$,
and moreover $\cf(C_\alpha(i))=f(i)$ for all nonzero $i\in\cf(\lambda)$.\footnote{
The above sort of \emph{club guessing} has been studied in \cite{sh:535},\cite{sh:g}
in relation with the theory of strong colorings (but with no
relation to $\square_\lambda$-sequences).
In an upcoming paper \cite{ostwalks}, we shall demonstrate that $\fsc_\lambda$ yields strong colorings in a simply definable way.}

Now, the main result of this paper reads as follows:
\begin{THMB}  $\square_\lambda$ implies $\sc_\lambda$, provided that:
\begin{itemize}
\item $\lambda$ is a limit uncountable cardinal, and $\lambda^\lambda=\lambda^+$;
\item $\lambda$ is a successor cardinal, and $\lambda^{<\lambda}<\lambda^\lambda=\lambda^+$.
\end{itemize}
\end{THMB}

In particular, $\gch$ implies that $\square_\lambda$ and $\sc_\lambda$ are equivalent,
for every uncountable cardinal $\lambda$.

We expect the above finding to admit many applications. In this paper, we present an application to the theory of trees,
obtaining the first example of an homogenous $\lambda^+$-Souslin tree,
for $\lambda$ which is a singular cardinal.
\begin{THMC} Suppose that $\square_\lambda$ holds for a given singular cardinal $\lambda$.

If $\lambda^{<\cf(\lambda)}<\lambda^\lambda=\lambda^+$,
then there exists an homogenous $\lambda^+$-Souslin tree,
which is moreover $\cf(\lambda)$-complete.
\end{THMC}

To conclude the introduction, we mention that our new principle is tightly related to
the well-known concept of ``square with a built-in diamond'' \cite{gray}, that unlike $\sc_\lambda$,
involves the existence of \emph{two} sequences --- a square sequence, and a diamond sequence.
We shall also consider here a parameterized version of this concept, and study its validity.
\begin{defn}\label{31} For an uncountable cardinal $\lambda$,
and a subset $\Gamma\s\lambda^+$, $\sd_\lambda^\Gamma$ asserts the existence of a $\square_\lambda$-sequence $\langle C_\alpha\mid\alpha<\lambda^+\rangle$,
and an additional sequence $\langle S_\alpha\mid \alpha<\lambda^+\rangle$ that interact in the following way:
\begin{enumerate}
\item if $\alpha\in\acc(\lambda^+)$ and  $\beta\in\acc(C_\alpha)$, then 
$S_\beta=S_\alpha\cap\beta$;
\item for every club $D\s\lambda^+$, every subset $A\s\lambda^+$, and every nonzero limit $\theta\in\Gamma$,
there exists some $\alpha<\lambda^+$ such that:
\begin{enumerate}
\item $C_\alpha\s D$;
\item $S_\alpha=A\cap\alpha$;
\item $\cf(\alpha)=\theta$;
\item $\sup(\acc(C_\alpha))=\alpha$.
\end{enumerate}
\end{enumerate}
\end{defn}

Define the variation $\sd_\lambda(S)$ in the obvious way. Then, it will be proved:

\begin{THMD} For an uncountable cardinal $\lambda$, the following are equivalent:
\begin{itemize}
\item $\square_\lambda+\ch_\lambda$;
\item $\sd^{\Gamma}_{\lambda}$, for $\Gamma=\reg(\lambda)$;
\item $\sd_\lambda(S)$, for every stationary  $S\s E^{\lambda^+}_{>\omega}\cap E^{\lambda^+}_{\not=\cf(\lambda)}$;
\item $\sd_\lambda(S)$, for every $S\s \lambda^+$ that reflects stationarily often.
\end{itemize}
\end{THMD}

Let us mention that in \cite{sh221}, Abraham, Shelah and Solovay  proved
that if $\ch_\lambda+\square_\lambda$ holds for a cardinal $\lambda$ which is \emph{singular strong limit},
then a certain approximation of $\sd^{\reg(\lambda)}_\lambda$, which they denote by $\sd(\lambda^+)$, holds. 
Thus,  as in the first example, the above theroem happens to apply to  all the relevant cardinals,
rather than just singular strong limits. This finding is somewhat unexpected, as previously such principles were known to be valid
only in the context of fine structure (for a very recent example, see \cite{Kypriotakis}).

Note also that Theorem D is optimal, as
the model of \cite{kingsteinhorn}
witnesses that $\textsf{GCH}+\square_{\lambda}$ is consistent with the failure of $\sd_{\lambda}^\Gamma$ for $\Gamma=\reg(\lambda^+)=\{\aleph_0,\aleph_1\}$,
hence we must indeed restrict ourselves to $\Gamma=\reg(\lambda)$.
On another front, forcing over $L$ with the poset from  \cite[$\S2$]{sh186} witnesses that $\textsf{GCH}+\square_{\aleph_\omega}$ is consistent with the failure
of $\sd_{\aleph_\omega}(S)$ for some non-reflecting stationary $S\s E^{\aleph_{\omega+1}}_\omega$,
hence we must restrict ourselves either to $S\s E^{\lambda^+}_{\not=\cf(\lambda)}$, or to reflecting stationary sets.
Finally, note that requirement (2)(d) of Definition \ref{31} put some obvious restrictions on stationary subsets of $E^{\lambda^+}_\omega$.

\subsection*{Organization of this paper} In Section 2, we study principles of square with built-in diamonds.
In Section 3, we study the Ostaszewski square. In Section 4, we provide a construction of an homogenous $\lambda^+$-Souslin tree
from a particular form of the Ostaszewski square which we denote by $\sc_{\lambda,\kappa}^{\Gamma,\mu}$.
Finally, in Section 5, we provide proofs for Theorems A--D,
based on the results of sections 2--4.

The paper is concluded with Section 6, in which we make some additional remarks,
and pose a few questions.

\subsection*{Notation and Conventions}
We abbreviate by  $\ch_\lambda$ the local Continuum Hypothesis for $\lambda$,
namely,  that $2^\lambda=\lambda^+$.
Denote $E^{\delta}_\kappa:=\{\alpha<\delta\mid \cf(\alpha)=\kappa\}$. The set $E^{\delta}_{>\kappa}$
is defined in a similar way.
Denote $\reg(\lambda):=\{ \alpha<\lambda\mid \cf(\alpha)=\alpha\ge\omega\}$.
For ordinals $\alpha,\beta$, we let $[\alpha,\beta):=\{ \gamma<\beta\mid \gamma\ge\alpha\}$.
We also define the ordinals-intervals $(\alpha,\beta]$ and $(\alpha,\beta)$ in a similar fashion.

For a set of ordinals, $A$, denote $\acc^+(A):=\{ \alpha<\sup(A)\mid \sup(A\cap\alpha)=\alpha\}$,
$\acc(A):=\acc^+(A)\cap A$, and $\nacc(A):=A\bks\acc(A)$.
If $i<\otp(A)$, we sometime let $A(i)$ denote the $i_{th}$ element of $A$.
The set $A$ is said to be \emph{closed} if $\acc(A)=\acc^+(A)$, it is a \emph{club in $\alpha$}
if it is closed and $A$ is a cofinal subset of $\alpha$.
It is \emph{stationary in $\alpha$} if it meets every club in $\alpha$.
Finally, $S\s\kappa$ is said to reflect stationarily often,
if $\{\alpha\in E^{\kappa}_{>\omega}\mid S\cap\alpha\text{ is stationary in }\alpha\}$
is stationary in $\kappa$.

A \emph{tree} is a partially ordered set $\langle T,<\rangle$ such that $x_{\downarrow}:=\{ y\in T\mid y< x\}$
is well-ordered by $<$ for all $x\in T$.
The \emph{height} of a node $x\in T$ is defined as $\h(x):=\otp(x_\downarrow,<)$.
The height of the whole tree is defined as $\h(T):=\sup\{ \h(x)\mid x\in  T\}$.
Denote $x^{\uparrow}:=\{ y\in T\mid x< y\}$. The tree is said to be \emph{homogenous} provided $x^\uparrow$ and $y^\uparrow$ are isomorphic
for every $x,y\in T$ with $\h(x)=\h(y)$.

A subset $C\s T$ is a \emph{chain} if $C$ is linearly-ordered by $<$.
The tree is said to be \emph{$\kappa$-complete} if for every chain $C\s T$ of size $<\kappa$,
there exists some $x\in T$ such that $C\s x_\downarrow$.
A subset $A\s T$ is an \emph{antichain}
if $x_\downarrow\cap y_\downarrow=\emptyset$ for all distinct $x,y\in A$.
Finally, a \emph{$\lambda^+$-Souslin tree} is a tree $\langle T,<\rangle$ of height $\lambda^+$
such that any $B\s T$ of size $\lambda^+$ is neither a chain, nor an antichain.

\npg\section{Square with built-in diamond}\label{sec2}
For brevity, we shall further say that $\langle (C_\alpha,S_\alpha)\mid \alpha<\lambda^+\rangle$
is a $\sd_\lambda^\Gamma$-sequence, if the two sequences $\langle C_\alpha\mid \alpha<\lambda^+\rangle$,
$\langle S_\alpha\mid \alpha<\lambda^+\rangle$, are as in Definition \ref{31}.

\begin{thm}\label{61} Suppose that $\lambda$ is an uncountable cardinal which is not the successor of a regular cardinal.

If $\square_\lambda+\ch_\lambda$ holds, then $\sd^{\Gamma}_{\lambda}$
is valid for some cofinal $\Gamma\s\reg(\lambda)$.
\end{thm}
\begin{proof} Fix a $\square_\lambda$-sequence, $\langle C_\alpha\mid \alpha<\lambda^+\rangle$
with  $C_{\alpha+1}=\emptyset$ for all $\alpha<\lambda^+$.

We shall reach our goal gradually, where we first define a $\square_\lambda$-sequence $\langle C_\alpha'\mid \alpha<\lambda^+\rangle$
that has a nice partition property with respect to some cofinal $\Gamma\s\reg(\lambda)$ (very much like in \cite{sh221}),
then we continue to define a $\square_\lambda$-sequence $\langle C^*_\alpha\mid \alpha<\lambda^+\rangle$
that guesses clubs, and finally we shall obtain a sequence $\langle C^\bullet_\alpha\mid \alpha<\lambda^+\rangle$
that is ready for $\sd^\Gamma_\lambda$.

Put $\kappa:=\sup(\reg(\lambda))$. Then $\lambda\in\{\kappa,\kappa^+\}$, and $\kappa$ is a limit ordinal with $\cf(\kappa)\le\cf(\lambda)$.
Let $\{ \kappa_i\mid i<\cf(\kappa)\}$ denote the increasing enumeration of some cofinal subset of $\kappa$.

We commence with defining a sequence  $\langle(\gamma_i,\lambda_i,T_i)\mid i<\cf(\kappa)\rangle$ by induction on $i<\cf(\kappa)$:

Let $\gamma_0:=\omega$, and $\gamma_{j}:=\sup_{i<j}\gamma_j$ for a limit $j$.
Now, suppose that $\gamma_i$ is defined for a given  $i<\cf(\kappa)$,
and let us define $\gamma_{i+1}$, as well as $\lambda_i$ and $T_i$.

Consider the regressive function $f_{i}:E^{\lambda^+}_{>\max\{\gamma_i,\kappa_i\}}\cap E^{\lambda^+}_{\not=\cf(\lambda)}\rightarrow\lambda\times\lambda$ satisfying:
$$f_{i+1}(\alpha):=(\cf(\alpha),\otp(C_\alpha)),$$
and then find a stationary set $T_i\s\dom(f_{i})$ and $\lambda_i,\gamma_{i+1}$ such that
 $f_i(\alpha)=(\lambda_i,\gamma_{i+1})$ for all $\alpha\in T_i$.
Note that $\gamma_i<\lambda_i\le\gamma_{i+1}<\lambda$.

Evidently, $\{\gamma_i\mid i<\cf(\kappa)\}$ is a club in $\kappa$,
and $\Gamma:=\{\lambda_i\mid i<\cf(\kappa)\}$ is a cofinal subset of $\reg(\lambda)$.
Pick a set of ordinals $\{ \gamma_i\mid \cf(\kappa)\le i<\cf(\lambda)\}\s[\kappa,\lambda)$ such that $\{ \gamma_i\mid i<\cf(\lambda)\}$ is a club in $\lambda$.
Denote $\gamma_{\cf(\lambda)}=\lambda$.

Denote $\Gamma_0:=\omega$. For all nonzero limit $i\le\cf(\lambda)$, denote $\Gamma_i:=\{\gamma_j\mid j<i\}$.
For all $i<\cf(\lambda)$, pick some club $\Omega_{i+1}\s\gamma_{i+1}$ of minimal order-type with $\min(\Omega_{i+1})=\gamma_i+1$.
Next, for all limit $\epsilon\le\lambda$, we define the club $E_\epsilon\s\epsilon$ by letting:
$$E_\epsilon:=\begin{cases}
\Gamma_i,&\epsilon=\gamma_i\ \&\ i\text{ is limit},\\
\Omega_{i+1}\cap\epsilon,&\epsilon\in(\gamma_i,\gamma_{i+1}]\ \&\ \epsilon\in\acc(\Omega_{i+1})\cup\{\gamma_{i+1}\}\\
\epsilon\bks\sup(\Omega_{i+1}\cap\epsilon),&\epsilon\in(\gamma_i,\gamma_{i+1})\ \&\ \epsilon\not\in\acc(\Omega_{i+1})
\end{cases}.$$

Now, given $C\s\lambda^+$, let $\epsilon_C:=\otp(C)$ and $\pi_C:\epsilon_C\rightarrow C$ denote the inverse collapse.
Next, for all limit $\alpha<\lambda^+$, put $C_\alpha':=\pi_{C_\alpha}``E_{\epsilon_{C_\alpha}}$.
Then $C_\alpha'$ is a subclub of $C_\alpha$ and in particular, $\acc(C_\alpha')\s\acc(C_\alpha)$.

\begin{claim}\label{321}\begin{enumerate}
\item if $\alpha\in E^{\lambda^+}_{<\lambda}$, then $\otp(C_\alpha')<\lambda$.
\item If $\alpha<\lambda^+$ and $\beta\in\acc(C_\alpha')$, then $C_\alpha'\cap\beta=C'_\beta$.
\end{enumerate}
\end{claim}
\begin{proof}
(1) As $\otp(C'_\alpha)\le\otp(C_\alpha)\le \lambda$ and $\cf(\otp(C'_\alpha))=\cf(\alpha)$ for all $\alpha<\lambda^+$,
if $\lambda$ is regular then we trivially get that $\otp(C'_\alpha)<\lambda$ for all $\alpha\in E^{\lambda^+}_{<\lambda}$.
Next, assume towards a contradiction that $\lambda$ is singular
and $\otp(C'_\alpha)=\lambda$ for some fixed $\alpha\in E^{\lambda^+}_{<\lambda}$.
It then follows $\otp(C_\alpha)=\lambda$,
and hence $\otp(C'_\alpha)=\otp(\pi_{C_\alpha}``\Gamma_i)=\otp(\Gamma)=\cf(\lambda)<\lambda$.
This is a contradiction.

(2)
Put $\beta':=\pi_{C_\alpha}^{-1}(\beta)$. Then $C_\alpha'\cap\beta=\pi_{C_\alpha}``(E_{\epsilon_{C_\alpha}}\cap\beta')$.
As $\beta\in\acc(C_\alpha')$, we have $\beta\in\acc(C_\alpha)$, hence, $C_\beta=C_\alpha\cap\beta$
and $\pi_{C_\beta}\s\pi_{C_\alpha}$. In particular, $C_\alpha'\cap\beta=\pi_{C_\beta}``(E_{\epsilon_{C_\alpha}}\cap\beta')$,
and since $\sup(C_\alpha'\cap\beta)=\beta$, we infer that $\dom(\pi_{C_\beta})=\beta'$.

Denote $\epsilon:=\epsilon_{C_\alpha}$ and $\varepsilon:=\beta'$, that is, $\varepsilon=\epsilon_{C_\beta}$.
Let us examine three cases:

$\blacktriangleright$  If $\epsilon=\gamma_i$ for some limit $i<\cf(\lambda)$, then $E_\epsilon=\Gamma_i$ for this $i$,
and so the fact that $\beta\in\acc(C_\alpha')$ implies that $\beta'\in\acc(\Gamma_i)$.
Pick a limit $j<i$ such that $\beta'=\gamma_j$,
then $E_{\varepsilon}=\Gamma_j$, and
$$C_\beta'=\pi_{C_\beta}``\Gamma_j=\pi_{C_\alpha}``\Gamma_j=\pi_{C_\alpha}``(\Gamma_i\cap\beta')=\pi_{C_\alpha}``(\Gamma_i)\cap\pi_{C_\alpha}(\beta')=C_\alpha'\cap\beta.$$

$\blacktriangleright$ if $\epsilon\in(\gamma_i,\gamma_{i+1}]\ \&\ \epsilon\in\acc(\Omega_{i+1})\cup\{\gamma_{i+1}\}$,
then $E_\epsilon=\Omega_{i+1}\cap\epsilon$ and $\beta'\in\acc(\Omega_{i+1})$.
In particular, $\varepsilon\in(\gamma_i,\gamma_{i+1}]$ and $\varepsilon\in\acc(\Omega_{i+1})$,
so $E_{\varepsilon}=\Omega_{i+1}\cap\varepsilon$, and
$$C_\beta'=\pi_{C_\beta}``(\Omega_{i+1}\cap\varepsilon)=\pi_{C_\alpha}``(\Omega_{i+1}\cap\epsilon)\cap\pi_{C_\alpha}(\varepsilon)=C'_\alpha\cap\beta.$$

$\blacktriangleright$ if $\epsilon\in(\gamma_i,\gamma_{i+1})\ \&\ \epsilon\not\in\acc(\Omega_{i+1})$,
then $E_\epsilon=\epsilon\bks\sup(\Omega_{i+1}\cap\epsilon)$, and hence $\varepsilon=\beta'$ is an
element of $(\sup(\Omega_{i+1}\cap\epsilon),\epsilon)$. So, $\varepsilon\not\in\Omega_{i+1}$,
and $\gamma_i+1=\min\Omega_{i+1}<\varepsilon<\epsilon\le\gamma_{i+1}$.

It follows that $E_\varepsilon=\varepsilon\bks(\Omega_{i+1}\cap\varepsilon)$,
and since $\varepsilon\in(\sup(\Omega_{i+1}\cap\epsilon,\epsilon)$, we actually get that $E_\varepsilon=\varepsilon\bks(\Omega_{i+1}\cap\epsilon)$.
Altogether, $C_\beta'=C_\alpha'\cap\beta$.
\end{proof}

For all $i<\cf(\kappa)$, pick a club $D_i\s\lambda^+$, such that for every club $D\s\lambda^+$,
there exists some $\alpha\in T_i\cap\acc(D_i)$ such that $C_\alpha'\cap D_i\s D$.%
\footnote{This is possible since $T_i\s E^{\lambda^+}_{>\omega}\cap E^{\lambda^+}_{<\lambda}$, and $\otp(C_\alpha')<\lambda$ for all $\alpha\in T_i$. See Theorem 2.17 in \cite{abraham-magidor}.}
Put $D^*:=\bigcap_{i<\cf(\kappa)}D_i$,
and let $\pi:=\pi_{D^*}$, that is, $\pi:\lambda^+\rightarrow D^*$ is the inverse collapse.

For all limit $\alpha<\lambda^+$,
let $c_\alpha$ be some club subset of $\alpha$ of minimal order-type, and put:
$$C_\alpha^*:=\begin{cases}\pi^{-1}``C_{\pi(\alpha)}',&\sup(C_{\pi(\alpha)}'\cap D^*)=\pi(\alpha)\\
c_\alpha,&\text{otherwise}\end{cases}.$$

Note that if $\alpha\in E^{\lambda^+}_{>\omega}$, then $\pi(\alpha)\cap D^*$ is a club in $\pi(\alpha)$, and hence the ``otherwise''
case can only occur for $\alpha\in E^{\lambda^+}_\omega$.

\begin{claim}\label{322} $\langle C^*_\alpha\mid \alpha<\lambda^+\rangle$ is a $\square_\lambda$-sequence.

Moreover, $\otp(C_\alpha^*)<\lambda$ for all $\alpha\in E^{\lambda^+}_{<\lambda}$.
\end{claim}
\begin{proof} Since $\pi$ is continuous, it is obvious that $C_\alpha^*$
is a club subset of $\alpha$ for all limit $\alpha<\lambda^+$.
We also have $\otp(C_\alpha^*)\le\otp(C_{\pi(\alpha)}')\le\otp(C_{\pi(\alpha)})$.
So, if $\cf(\alpha)<\lambda$, then $\otp(C_\alpha^*)<\lambda$, and otherwise $\otp(C_\alpha^*)=\cf(\pi(\alpha))=\lambda$.

Next, suppose that $\alpha<\lambda^+$ is a limit ordinal and that $\beta\in\acc(C_\alpha^*)$.

$\blacktriangleright$ If $\sup(C_{\pi(\alpha)}'\cap D^*)=\pi(\alpha)$, then $\pi(\beta)\in\acc(C'_{\pi(\alpha)}\cap D^*)$. That is:
\begin{itemize}
\item $\sup(C_{\pi(\beta)}'\cap D^*)=\pi(\beta)$, and $C^*_\beta=\pi^{-1}``C'_{\pi(\beta)}$;
\item $C_{\pi(\alpha)}'\cap\pi(\beta)=C_{\pi(\beta)}'$, and hence
\item $C_\alpha^*\cap\beta=C_\beta^*$.
\end{itemize}

$\blacktriangleright$ If $\sup(C_{\pi(\alpha)}'\cap D^*)<\pi(\alpha)$,
then  $\cf(\alpha)=\omega$ and $\otp(C^*_\alpha)=\otp(c_\alpha)=\omega$. In particular, $\acc(C_\alpha^*)=\emptyset$,
and this case is of no interest.
\end{proof}

\begin{claim}\label{433}
For every club $D\s\lambda^+$ and every $i<\cf(\kappa)$,
the set $$\{ \alpha\in T_i\mid C_\alpha^*\s C_\alpha'\ \&\ C_\alpha^*\s D\}$$ is stationary.
\end{claim}
\begin{proof} Suppose that $D,E$ are club subsets of $\lambda^+$,
and that $i<\cf(\kappa)$. Our aim is to find an $\alpha\in T_i$ such that $C_\alpha^*\s C_\alpha'$, $C_\alpha^*\s D$, and $\alpha\in E$.

Consider the next club:
$$D':=\{\alpha\in D\cap E\mid \pi(\alpha)=\alpha\}.$$

By the choice of $D_i$,
we may pick some $\alpha\in T_i\cap\acc(D_i)$ such that $C_\alpha'\cap D_i\s D'$.
By $\alpha\in T_i$, we get that $\cf(\alpha)=\lambda_i>\gamma_0=\omega$,
so the fact that $\alpha\in\acc(D_i)$, implies that $\sup(C_\alpha'\cap D_i)=\alpha$.
Since $C_\alpha'\cap D_i\s D'$, we get that $\alpha\in\acc(D')\s D'$.
It follows that:
\begin{itemize}
\item $\alpha\in T_i\cap E$, because $E\s D'$;
\item $C_\alpha^*=\pi^{-1}``C_{\pi(\alpha)}'$, because $\cf(\alpha)>\omega$;
\item $\pi^{-1}``C_{\pi(\alpha)}'=\pi^{-1}C'_{\alpha}$, because $\pi(\alpha)=\alpha$;
\item $\pi^{-1}C_\alpha'=C_\alpha'\cap D^*$, because $\pi(\beta)=\beta$ for all $\beta\in D'\supseteq C_\alpha'\cap D_i$.
\end{itemize}

Altogether, $C^*_\alpha=C_\alpha'\cap D^*\s C_\alpha'\cap D_i\s D'$, and hence $\alpha$ is as requested.
\end{proof}

For all $i<\cf(\kappa)$, let $S_i:=\bigcup\{\acc(C_\alpha^*)\cup\{\alpha\}\mid \alpha\in T_i\ \&\ C_\alpha^*\s C_\alpha'\}$.

\begin{claim}\label{432} For all  $i<j<\cf(\kappa)$, $S_i\cap S_j=\emptyset$.
\end{claim}
\begin{proof} Suppose not. Pick $\alpha_i\in T_i$, $\alpha_j\in T_j$ such that
$(\acc(C_{\alpha_i}^*)\cup\{\alpha_{i}\})\cap (\acc(C_{\alpha_j}^*)\cup\{\alpha_{j}\})\not=\emptyset$,
and let $\beta$ be some element of this nonempty set.
In particular, $\beta\in\acc(C_{\alpha_i})\cup\{\alpha_i\}$, and hence $\cf(\beta)\le\otp(C_\beta)\le\otp(C_{\alpha_i})=\gamma_{i+1}$.\footnote{%
Recall that $\otp(C_\alpha)=\gamma_{i+1}$ and $\cf(\alpha)=\lambda_i$ for all $\alpha\in T_i$.}
Since $\cf(\alpha_j)=\lambda_j>\gamma_{i+1}$, we get that $\beta\not=\alpha_j$,
and so it must be the case that $\beta\in\acc(C'_{\alpha_j})$.
As $C_\beta\s C_{\alpha_j}$ and $\otp(C_\beta)\le\gamma_{i+1}$, we get that $\pi_{C_{\alpha_j}}^{-1}(\beta)\le\gamma_{i+1}$.
By $C_{\alpha_j}'=\pi_{C_{\alpha_j}}``E_{\epsilon_{C_{\alpha_j}}}$ and $\beta\in\acc(C'_{\alpha_j})$, this means that $\gamma_{i+1}\cap E_{\epsilon_{C_{\alpha_j}}}\not=\emptyset$.
However, $\alpha_j\in T_j$, and hence $\epsilon_{C_{\alpha_j}}=\gamma_{j+1}$,
so $E_{\epsilon_{C_{\alpha_j}}}=\Omega_{j+1}$, and $\min(E_{\epsilon_{C_{\alpha_j}}})=\gamma_{j}+1>\gamma_j\ge\gamma_{i+1}$.
This is a contradiction.
\end{proof}

\begin{claim}\label{435} For every limit $\alpha<\lambda^+$, the set
$$\{ i<\cf(\kappa)\mid (\acc(C_\alpha^*)\cup\{\alpha\})\cap S_i\not=\emptyset\}$$
contains at most a single element.
\end{claim}
\begin{proof} Assume indirectly that there exist $\beta\le\gamma$ in $\acc(C^*_\alpha)\cup\{\alpha\}$,
 and $i\not=j$, such that $\beta\in S_j, \gamma\in S_i$. As $S_i\cap S_j=\emptyset$, we actually have $\beta<\gamma$.
\begin{itemize}
\item As $\gamma\in\acc(C^*_\alpha)\cup\{\alpha\}$, and $\beta\in\acc(C^*_\alpha\cap\gamma)$, we get that $\beta\in\acc(C^*_\gamma)$.

\item As $\gamma\in S_i$, we may find some $\delta\in T_i$ such that $\gamma\in\acc(C_\delta^*)\cup\{\delta^*\}$.
So $C_\gamma^*=C_\delta^*\cap\gamma$, and $\acc(C_\gamma^*)=\acc(C_\delta^*)\cap\gamma$.
\end{itemize}

Altogether  $\beta\in\acc(C_\delta^*)$ for some $\delta\in T_i$,
and hence $\beta\in S_i$. This contradicts the fact that $\beta\in S_j$.
\end{proof}

\begin{claim}\label{436} There exists a matrix $\mathcal H=\langle H_\alpha^j\mid j<\cf(\lambda), \alpha<\lambda^+\rangle$
such that for every limit $\alpha<\lambda^+$:
\begin{enumerate}
\item $\{ H_\alpha^j\mid j<\cf(\lambda)\}\s[\alpha\times\alpha]^{<\lambda}$
is an increasing chain converging to $\alpha\times\alpha$;
\item if $\beta\in\acc(C^*_\alpha)$, then $H^j_\beta=H^j_\alpha\cap(\beta\times\beta)$ for all $j<\cf(\lambda)$.
\end{enumerate}
\end{claim}
\begin{proof} By \cite{todorcevic}, we may pick a sequence of injections $\langle \rho_\alpha:\alpha\rightarrow{}^{<\omega}\lambda\mid \alpha<\lambda^+\rangle$
such that if $\beta\in\acc(C^*_\alpha)$, then $\rho_\beta=\rho_\alpha\restriction\beta$.
Let $\langle \chi_j\mid j<\cf(\lambda)\rangle$ be an increasing sequence of ordinals,
converging to $\lambda$. Then, put $$H^j_\alpha:=(\rho_\alpha^{-1}[{}^{<\omega}\chi_j])^2,\quad(j<\cf(\lambda), \alpha<\lambda^+).$$
\end{proof}

Let $\mathcal H$ be given by the preceding claim.
By $\ch_\lambda$, let $\{X_\gamma\mid \gamma<\lambda^+\}$ be an enumeration of $[\lambda\times\lambda\times\lambda^+]^{\le\lambda}$.
For all $(j,\tau)\in\lambda\times\lambda$ and $X\s\lambda\times\lambda\times\lambda^+$,
let $\pi_{j,\tau}(X):=\{\varsigma<\lambda^+\mid (j,\tau,\varsigma)\in X\}$.
\begin{claim}\label{437} Suppose that $i<\cf(\kappa)$.

There exist $(j,\tau)\in\cf(\lambda)\times\lambda$ and $Y\s\lambda^+\times\lambda^+$
such that for every club $D\s\lambda^+$ and every subset $Z\s\lambda^+$, there exists some  $\alpha\in T_i$ such that:
\begin{enumerate}
\item $C_\alpha^*\s C_\alpha'\cap D$;
\item $H^j_\alpha\bks Y\s \{(\eta,\gamma)\mid Z\cap\eta=\pi_{j,\tau}(X_\gamma)\}$;
\item $\sup\{\eta\in{C^*_\alpha}\mid (\eta,\gamma)\in H^j_\alpha\bks Y\text{ for some }\gamma\}=\alpha$.
\end{enumerate}
\end{claim}
\begin{proof} Suppose not, and let us argue as in \cite{rinot07}, which itself is based on \cite{sh922}.
We build by recursion on $\tau<\lambda$, three sequences:
\begin{enumerate}
\item[(\textrm{I})] $\langle \{Z^j_\tau\mid j<\cf(\lambda)\}\mid \tau<\lambda\rangle$;
\item[(\textrm{II})] $\langle \{D^j_\tau\mid j<\cf(\lambda)\}\mid \tau<\lambda\rangle$;
\item[(\textrm{III})] $\langle \{Y^j_\tau\mid j<\cf(\lambda)\} \mid \tau<\lambda\rangle$.
\end{enumerate}

Base case, $\tau=0$. Fix $j<\cf(\lambda)$. By the hypothesis, we may find a set $Z^j_0\s\lambda^+$,
and a club $D^j_0\s\lambda^+$ such that for all $\alpha\in T_i$,
at least one of the following fails:
\begin{enumerate}
\item $C_\alpha^*\s C_\alpha'\cap D^j_0$;
\item $H^j_\alpha\s \{(\eta,\gamma)\mid Z^j_0\cap\eta=\pi_{j,0}(X_\gamma)\}$;
\item $\sup\{\eta\in{C^*_\alpha}\mid (\eta,\gamma)\in H^j_\alpha \text{ for some }\gamma\}=\alpha$.
\end{enumerate}

Put $$Y^j_0:=\{ (\eta,\gamma)\in \lambda^+\times\lambda^+\mid Z^j_0\cap\eta\not=\pi_{j,0}(X_\gamma)\}.$$

Next, assume that the three sequences are defined up to some $\tau<\lambda$.
Fix $j<\cf(\lambda)$. By the hypothesis, we may find a set $Z^j_\tau\s\lambda^+$,
and a club $D^j_\tau\s\lambda^+$ such that for all $\alpha\in T_i$, at least one of the following fails:
\begin{enumerate}
\item $C_\alpha^*\s C_\alpha'\cap D^j_\tau$;
\item $H^j_\alpha\bks \bigcup_{\iota<\tau}Y^j_\iota\s \{(\eta,\gamma)\mid Z^j_\tau\cap\eta=\pi_{j,\tau}(X_\gamma)\}$;
\item $\sup\{\eta\in{C^*_\alpha}\mid (\eta,\gamma)\in H^j_\alpha\bks \bigcup_{\iota<\tau}Y^j_\iota \text{ for some }\gamma\}=\alpha$.
\end{enumerate}

Put $$Y^j_\tau:=\{ (\eta,\gamma)\in \lambda^+\times\lambda^+\mid Z^j_\tau\cap\eta\not=\pi_{j,\tau}(X_\gamma)\}.$$
This completes the recursive construction.

To meet a contradiction, let $Z:=\{ (j,\tau,\varsigma)\mid j<\cf(\lambda),\tau<\lambda,\varsigma\in Z^j_\tau\}$ and
define a function $f:\lambda^+\rightarrow\lambda^+$ by letting:
$$f(\eta):=\min\{\gamma<\lambda^+\mid Z\cap(\lambda\times\lambda\times\eta)=X_\gamma\},\quad(\eta<\lambda^+).$$

Put $D:=\{ \alpha\in\bigcap_{\tau<\lambda}\bigcap_{j<\cf(\lambda)}D^j_\tau\mid f``\alpha\s \alpha\}$.
By Claim \ref{433}, let us pick some $\alpha\in T_i$ such that $C_\alpha^*\s C_\alpha'\cap D$.

As $f\restriction{C^*_\alpha}\s\alpha\times\alpha$, we may define a function $g:{C^*_\alpha}\rightarrow\cf(\lambda)$ by letting;
$$g(\eta):=\min\{ j<\cf(\lambda)\mid (\eta,f(\eta))\in H^j_\alpha\},\quad(\eta\in{C^*_\alpha}).$$
As $\alpha\in T_i$, we have $\cf(\alpha)=\lambda_i\not=\cf(\lambda)$,
so let us fix some $j<\cf(\lambda)$, such that $L_j:=g^{-1}\{j\}$ is cofinal in $\alpha$.
As $f\restriction L_j\s H^j_\alpha$,
we have $(f\restriction L_j)\cap Y^j_\tau=\emptyset$ for all $\tau<\lambda$.
Since $\sup(L_j)=\alpha$, we infer that for every $\tau<\lambda$ and $j<\cf(\lambda)$,
the counterexample $(Z^j_\tau,D^j_\tau)$ is reflected at the $\alpha$ level, as the failure of clause (2).
Namely,  $\langle H_\alpha^j\bks\bigcup_{\iota<\tau}Y^j_\iota\mid 0<\tau<\lambda\rangle$
is a strictly decreasing sequence of non-empty subsets of $H^j_\alpha$.
However, this is a contradiction  to the fact that $|H^j_\alpha|<\lambda$.
\end{proof}

Let $\langle (j_i,\tau_i,Y_i)\mid i<\cf(\lambda)\rangle$ be given by the previous claim.
For all limit $\alpha<\lambda^+$, and $i<\cf(\lambda)$, let:
$$f_\alpha^i:=\{(\eta,\gamma)\in H^{j_i}_\alpha\bks Y_i\mid ((\eta,\gamma')\in H^{j_i}_\alpha\bks Y_i)\Longrightarrow \gamma'\ge\gamma\};$$
$$A^i_\alpha(\delta):=\{ \eta\in \dom(f^i_\alpha)\cap {C^*_\alpha}\cap\delta\mid f^i_\alpha(\eta)<\delta\},\quad(\delta<\alpha).$$
Then let
 $C^i_\alpha$ be the set of all $\delta<\alpha$ such that the following holds:
\begin{enumerate}
\item $\sup(A^i_\alpha(\delta))=\delta$;
\item if $\eta=\min(A^i_\alpha(\delta))$, then $\pi_{j_i,\tau_i}(X_{f^i_\alpha(\eta)})\s\eta$;
\item if $\eta'<\eta$ are in $A^i_\alpha(\delta)$, then $\pi_{j_i,\tau_i}(X_{f^i_\alpha(\eta)})\bks \pi_{j_i,\tau_i}(X_{f^i_\alpha(\eta')})\s[\eta',\eta)$.
\end{enumerate}

Note that $C^i_\alpha$ is a closed subset of $C^*_\alpha$. We also have coherence:

\begin{claim}\label{438} If $\beta\in\acc(C^*_\alpha)$, then:
\begin{itemize}
\item $C^i_{\beta}=C^i_\alpha\cap\beta$;
\item $f^i_{\beta}\restriction A^i_\beta(\delta)=f^i_{\alpha}\restriction A^i_\alpha(\delta)$ for all $\delta<\beta$.
\end{itemize}
\end{claim}
\begin{proof} If $\beta\in\acc(C^*_\alpha)$, then by Claim \ref{436}, we have $H^{j_i}_{\beta}=H^{j_i}_\alpha\cap(\beta\times\beta)$,
and hence $f^i_{\beta}=f^i_\alpha\restriction\{\eta\in \beta\mid f^i_\alpha(\eta)<\beta\}$. Also, $C^*_\beta=C^*_\alpha\cap\beta$,
so $A^i_\alpha(\delta)=A^i_\beta(\delta)$ for all $\delta<\beta$. Consequently, $C^i_{\beta}=C^i_\alpha\cap\beta$ and
$f^i_{\beta}\restriction A^i_\beta(\delta)=f^i_{\alpha}\restriction A^i_\alpha(\delta)$ for all $\delta<\beta$.
\end{proof}

Recalling Claim \ref{435}, we are now ready to define our $\sd^\Gamma_\lambda$-sequence. For all limit $\alpha<\lambda^+$, put:
$$C_\alpha^\bullet:=\begin{cases}C_\alpha^*,&(\acc(C_\alpha^*)\cup\{\alpha\})\cap\bigcup_{i<\cf(\kappa)}S_i=\emptyset\\
C_\alpha^i\cup (C_\alpha^*\bks\sup(C_\alpha^i)),&(\acc(C_\alpha^*)\cup\{\alpha\})\cap S_i\not=\emptyset
\end{cases}$$

and
$$Z_\alpha:=\begin{cases}\emptyset,&(\acc(C_\alpha^*)\cup\{\alpha\})\cap\bigcup_{i<\cf(\kappa)}S_i=\emptyset\\
\bigcup\{\pi_{j_i,\tau_i}(X_{f^i_\alpha(\eta)})\mid \eta\in A^i_\alpha(\delta),\delta\in C^i_\alpha\},&(\acc(C_\alpha^*)\cup\{\alpha\})\cap S_i\not=\emptyset
\end{cases}$$

So $C_\alpha^\bullet$ is a subclub of $C_\alpha^*$, and $Z_\alpha\s\alpha$.
In particular, $\otp(C_\alpha^\bullet)<\lambda$ whenever $\alpha\in E^{\lambda^+}_{<\lambda}$.

\begin{claim} $\langle (C_\alpha^\bullet,Z_\alpha)\mid\alpha<\lambda^+\rangle$ is a $\sd^\Gamma_\lambda$-sequence.
\end{claim}
\begin{proof} First, suppose that $\beta\in\acc(C^\bullet_\alpha)$, and let us show that $C_\beta^\bullet=C_\alpha^\bullet\cap\beta$ and $S_\beta=S_\alpha\cap\beta$.
Note that $\beta\in\acc(C^*_\alpha)$, and hence $C^*_\beta=C^*_\alpha\cap\beta$.

We now consider two cases:

$\blacktriangleright$ If $(\acc(C_\alpha^*)\cup\{\alpha\})\cap\bigcup_{i<\cf(\kappa)}S_i=\emptyset$,
then $(C_\alpha^\bullet,Z_\alpha)=(C_\alpha^*,\emptyset)$ and $(C^*_\beta\cup\{\beta\})\cap\bigcup_{i<\cf(\kappa)}S_i=\emptyset$,
so $C_\beta^\bullet=C_\beta^*=C^*_\alpha\cap\beta=C^\bullet_\alpha\cap\beta$, and $Z_\beta=\emptyset=Z_\alpha\cap\beta$.

$\blacktriangleright$ If $(\acc(C_\alpha^*)\cup\{\alpha\})\cap S_i\not=\emptyset$,
then there exists some $\delta\in T_i$ such that $\alpha\in(\acc(C^*_\delta)\cup\{\delta\})\s S_i$.
In particular, $C^*_\alpha=C^*_\delta\cap\alpha$, so $\beta\in\acc(C^*_\delta)$,
and $\beta\in S_i$. It follows that:
\begin{itemize}
\item $C^\bullet_\alpha=C_\alpha^i\cup (C_\alpha^*\bks\sup(C_\alpha^i))$, and $C^\bullet_\beta=C_\beta^i\cup (C_\beta^*\bks\sup(C_\beta^i))$;
\item $Z_\alpha=\bigcup\{\pi_{j_i,\tau_i}(X_{f^i_\alpha(\eta)})\mid \eta\in A^i_\alpha(\delta),\delta\in C^i_\alpha\}$, and
\item $Z_\beta=\bigcup\{\pi_{j_i,\tau_i}(X_{f^i_\beta(\eta)})\mid \eta\in A^i_\beta(\delta),\delta\in C^i_\beta\}$.
\end{itemize}
Note that by Claim \ref{438}, we get that $C^i_\beta=C^i_\alpha\cap\beta$.
If $\sup(C^i_\alpha)\ge\beta$, then $C_\alpha^\bullet\cap\beta=C^i_\alpha\cap\beta=C^i_\beta=C^\bullet_\beta$.
If $\sup(C^i_\alpha)<\beta$, then letting $\gamma:=\sup(C^i_\alpha)$, we get that
$C^\bullet_\alpha=(C^i_\alpha\cap\beta)\cup(C^*_\alpha\bks\gamma)=C^i_\beta\cup(C^*_\alpha\bks\gamma)$,
and hence $C^\bullet_\alpha\cap\beta=C^i_\beta\cup(C^*_\beta\bks\gamma)=C^i_\beta\cup(C^*_\beta\bks\sup(C^i_\beta))=C^\bullet_\beta$.

By clauses (2) and (3) of the definition of $A^i_\alpha(\delta)$,
we get that for all $\delta<\alpha$:
 $$Z_\alpha\cap\beta=\bigcup\{\pi_{j_i,\tau_i}(X_{f^i_\alpha(\eta)})\mid \eta\in A^i_\alpha(\delta),\delta\in C^i_\alpha\cap\beta\}.$$
By $\beta\in\acc(C^*_\alpha)$, and Claim \ref{438}(1), we have:
$$Z_\alpha\cap\beta=\bigcup\{\pi_{j_i,\tau_i}(X_{f^i_\alpha(\eta)})\mid \eta\in A^i_\alpha(\delta),\delta\in C^i_\beta\}.$$
By Claim \ref{438}(2), $f^i_{\beta}\restriction A^i_\beta(\delta)=f^i_{\alpha}\restriction A^i_\alpha(\delta)$ for all $\delta<\beta$, and hence:
$$Z_\alpha\cap\beta=\bigcup\{\pi_{j_i,\tau_i}(X_{f^i_\beta(\eta)})\mid \eta\in A^i_\beta(\delta),\delta\in C^i_\beta\}=Z_\beta.$$

Finally, suppose that we are given a subset $Z\s\lambda^+$, a club $D\s\lambda^+$,
and $\theta\in\Gamma$, and let find some $\alpha\in E^{\lambda^+}_{\theta}$ such that  $C_\alpha^\bullet\s D$ and $Z\cap\alpha=Z_\alpha$.

Let $i<\cf(\kappa)$ be such that $\theta=\lambda_i$.
By the choice of $(j_i,\tau_i,Y_i)$ (as given by Claim \ref{437}),
there exists some  $\alpha\in T_i$ such that:
\begin{enumerate}
\item $C_\alpha^*\s C_\alpha'\cap D$;
\item $H^{j_i}_\alpha\bks Y_i\s \{(\eta,\gamma)\mid Z\cap\eta=\pi_{j_i,\tau_i}(X_\gamma)\}$;
\item $\sup\{\eta\in{C^*_\alpha}\mid (\eta,\gamma)\in H^{j_i}_\alpha\bks Y_i\text{ for some }\gamma\}=\alpha$.
\end{enumerate}

By (1), we get that $C^\bullet_\alpha\s C^*_\alpha\s D$. By $\cf(\alpha)=\lambda_i>\omega$, we get that $\sup(\acc(C_\alpha^\bullet))=\alpha$.
By (3), we get that $\sup(\dom(f^i_\alpha\cap C^*_\alpha))=\alpha$. As $\cf(\alpha)>\omega$,
the set $e_\alpha=\{\delta<\alpha\mid  \sup(A^i_\alpha(\delta))=\delta\}$ is a club in $\alpha$.
Now, if $\delta\in e_\alpha$, then by clause (2), we get that for all $\eta\in A^i_\alpha(\delta)$, we have $Z\cap\eta=\pi_{j_i,\tau_i}(X_{f^i_\alpha(\eta)})$.
Hence, $C^i_\alpha=e_\alpha$, $C^\bullet_\alpha=C^i_\alpha$, and:
$$Z\cap\alpha=\bigcup\{ Z\cap\eta\mid \eta\in A^i_\alpha(\delta), \delta\in C^i_\alpha\}=\bigcup\{\pi_{j_i,\tau_i}(X_{f^i_\alpha(\eta)})\mid \eta\in A^i_\alpha(\delta),\delta\in C^i_\alpha\}.$$
\end{proof}
This completes the proof.
\end{proof}

Next, consider the following ad-hoc variation of $\sd_\lambda(S)$:
\begin{defn} For an uncountable cardinal $\lambda$,
$\sd'_\lambda(S)$ asserts the existence of a sequence $\langle (C_\alpha,S_\alpha)\mid\alpha<\lambda^+\rangle$ such that:

\begin{enumerate}
\item $C_\alpha$ is a club subset of $\alpha$ for all limit $\alpha<\lambda^+$,
with  $\otp(C_\alpha)\le\lambda$;
\item if $\beta\in\acc(C_\alpha)$, then $C_\beta=C_\alpha\cap\beta$ and $S_\beta=S_\alpha\cap\beta$;
\item for every club $D\s\lambda^+$, every subset $A\s\lambda^+$,
there exists some $\alpha<\lambda^+$ such that:
\begin{enumerate}
\item $C_\alpha\s D$;
\item $S_\alpha=A\cap\alpha$;
\item $\alpha\in S$;
\item $\sup(\acc(C_\alpha))=\alpha$.
\end{enumerate}
\end{enumerate}
\end{defn}

So, the only difference between $\sd_\lambda'(S)$ and $\sd_\lambda(S)$
is that in the latter, guesses takes place within $S$, whereas, in the former,
there exists some stationary subset $S'\s S$ with $S'\cap\acc(C_\alpha)=\emptyset$ for all $\alpha<\lambda^+$,
on which the guesses takes place.

\begin{thm}\label{0034} Suppose that $\square_\lambda+\ch_\lambda$ holds for a given uncountable cardinal $\lambda$. Then:
\begin{enumerate}
\item $\sd'_\lambda(T)$ holds for every stationary  $T\s E^{\lambda^+}_{>\omega}\cap E^{\lambda^+}_{\not=\cf(\lambda)}$;
\item $\sd'_\lambda(S)$ holds for every $S\s \lambda^+$ that reflects stationarily often.
\end{enumerate}
\end{thm}
\begin{proof} A careful look at the proof of Theorem \ref{61} reveals that the hypotheses implies clause $(1)$.
Hence, we turn to deal with the second clause.

Suppose that $S\s \lambda^+$ reflects stationarily often.
That is, we assume that the set $T:=\{\alpha\in E^{\lambda^+}_{>\omega}\mid S\cap\alpha\text{ is stationary}\}$ is stationary.

$\blacktriangleright$ If $S\cap E^{\lambda^+}_{>\omega}\cap E^{\lambda^+}_{\not=\cf(\lambda)}$ is stationary, then $\sd'_\lambda(S)$ holds
as a consequence of clause (1), and we are done.

$\blacktriangleright$ If $S\s E^{\lambda^+}_{\cf(\lambda)}$ and $S$ reflects stationarily often, then $\lambda$ is singular and $T\s  E^{\lambda^+}_{>\omega}\cap E^{\lambda^+}_{\not=\cf(\lambda)}$ is stationary.
It follows that $\sd'_\lambda(T)$ holds. Let $\langle (C_\alpha,S_\alpha)\mid \alpha<\lambda^+\rangle$
witness the validity of $\sd'_\lambda(T)$; it is easy to see that this sequence also witnesses the validity of $\sd'_\lambda(S)$.

$\blacktriangleright$ Suppose that $S\s E^{\lambda^+}_\omega\cap E^{\lambda^+}_{\not=\cf(\lambda)}$,
and that $S$ reflects stationarily often.
Let $\langle C_\alpha\mid \alpha<\lambda^+\rangle$
be a $\square_{\lambda}$-sequence such that $\otp(C_\alpha)<\lambda$ for all $\alpha\in E^{\lambda^+}_{<\lambda}$ (e.g., as in Claim \ref{322}).
Given a club $E\s\lambda^+$, denote $\drop(\alpha,E):=\{ \sup(E\cap\beta)\mid \beta\in C_\alpha\bks \min(E)\}$.
We now reproduce Shelah's club-guessing argument for stationary sets that concentrates on countable cofinality:
\begin{claim}\label{351} There exists a club $D^*\s\lambda^+$ such that for every club $D\s\lambda^+$,
there exists some $\alpha\in S$ such that:
\begin{itemize}
\item $\drop(\alpha,D^*)\s D$, and
\item $\sup(\acc(\drop(\alpha,D^*))=\alpha$.
\end{itemize}
\end{claim}
\begin{proof} Suppose not. We shall define a sequence of club subsets of $\lambda^+$, $\langle D_i\mid i<\lambda\rangle$,
by recursion on $i<\lambda$. We start with picking  a club $D_0$
such that for all $\alpha\in S$, either $\sup(\acc(\drop(\alpha,\lambda^+))<\alpha$ or $\drop(\alpha,\lambda^+)\not\s D_0$.
Now suppose that $i<\lambda^+$ and that $\langle D_j\mid j<i\rangle$ have already been defined.
By the indirect assumption, we now pick a club $D_i\s\lambda^+$ 
such that $\sup(\acc(\drop(\alpha,\bigcap_{j<i}D_j))<\alpha$ or $\drop(\alpha,\bigcap_{j<i}D_j)\not\s D_i$,
for all $\alpha\in S$. 
This completes the construction.

Put $D^*:=\bigcap_{i<\lambda}D_i$, and pick $\delta\in\acc(D^*)\cap T$.
It is easy to see that $\acc(\drop(\delta,D^*))$ is a club subset of $\acc(C_\delta)$.
As $S\cap\delta$ is stationary in $\delta$, let us pick some $\alpha\in S\cap\acc(\drop(\delta,D^*))$.
Then, $\drop(\alpha,D^*)=\drop(\delta,D^*)\cap\alpha$
and $\sup(\acc(\drop(\alpha,D_i))=\alpha$ for all $i<\lambda$.
Consequently, $\drop(\alpha,\bigcap_{j<i}D_j)\not\s D_{i}$ for all $i<\lambda^+$.

As $\langle \bigcap_{j<i}D_j\mid i<\lambda\rangle$ is a decreasing chain,
for all $\beta\in C_\alpha$, the sequence of ordinals $\langle \sup((\bigcap_{j<i}D_j)\cap\beta)\mid i<\lambda\rangle$ stabilizes at some $i_\beta<\lambda$.
Since $\otp(C_\alpha)<\lambda$, we may fix a large enough $i^*<\lambda$ such that $i^*>i_\beta$ for all $\beta\in C_\alpha$.
As $\drop(\alpha,\bigcap_{j<i^*}D_j)\not\s D_{i^*}$, let us   pick some $\beta\in C_\alpha$
such that $\sup((\bigcap_{j<i^*}D_j)\cap\beta)\not\in D_{i^*}$.
As $\{ \sup((\bigcap_{j<i}D_j)\cap\beta)\mid i_\beta<i<\lambda\}$ is a singleton,
say it equals $\{\beta'\}$, we have $\beta'\in\bigcap_{j<i}D_j$ for all $j\in(i_\beta,\lambda)$.
So $\beta'\in D_{i^*}$, contradicting the fact that  $\sup((\bigcap_{j<i^*}D_j)\cap\beta)\not\in D_{i^*}$.
\end{proof}
Let $D^*$ be the club provided by the previous claim,
and let $\pi:\lambda^+\rightarrow D^*$ denote  the inverse collapse.
For all limit $\alpha<\lambda^+$, let $$C_\alpha^*:=\pi^{-1}``\drop(\pi(\alpha),D^*).$$

\begin{claim}\label{232}  $\langle C^*_\alpha\mid \alpha<\lambda^+\rangle$ is a $\square_\lambda$-sequence.

Moreover, $\otp(C_\alpha^*)<\lambda$ for all $\alpha\in E^{\lambda^+}_{<\lambda}$.
\end{claim}
\begin{proof} Fix a limit $\alpha<\lambda^+$. Since $\pi(\alpha)\in\acc(D^*)$, we get that $\drop(\pi(\alpha),D^*)$
is some club subset of $D^*\cap\pi(\alpha)$.
Since $\pi$ is continuous, we infer that $C_\alpha^*$ is a club subset of $\alpha$.
Also, it is clear that $\otp(C_\alpha^*)=\otp(\drop(\pi(\alpha),D^*))\le\otp(C_{\pi(\alpha)})$.
In particular, if $\cf(\alpha)<\lambda$, then $\cf(\pi(\alpha))<\lambda$, and $\otp(C_\alpha^*)<\lambda$.

Next, suppose that $\alpha<\lambda^+$ is a limit ordinal and that $\beta\in\acc(C_\alpha^*)$.
Then $\pi(\beta)\in\acc(\drop(\pi(\alpha),D^*))$, and hence $\pi(\beta)\in\acc(C'_{\pi(\alpha)})\cap\acc(D^*)$.
So $C'_{\pi(\beta)}=C'_{\pi(\alpha)}\cap\pi(\beta)$, and $\drop(\pi(\beta),D^*)=\drop(\pi(\alpha),D^*)\cap\pi(\beta)$.
By pulling back the latter, using $\pi$, we infer that $C_\beta^*=C_\alpha^*\cap\beta$.
\end{proof}

\begin{claim} For every club $D\s\lambda^+$,
the set $$\{ \alpha\in S\mid \sup(\acc(C_\alpha^*))=\alpha\ \&\ C_\alpha^*\s D\}$$ is stationary.
\end{claim}
\begin{proof} Suppose that $D,E$ are club subsets of $\lambda^+$.
Our aim is to find an $\alpha\in S$ such that $\sup(\acc(C_\alpha^*))=\alpha$, $C_\alpha^*\s D$, and $\alpha\in E$.

Consider the next club:
$$D':=\{\alpha\in D\cap E\mid \pi(\alpha)=\alpha\}.$$

By the choice of $D^*$,
we may pick some $\alpha\in S$ such that $\drop(\alpha,D^*)\s D'$
and $\sup(\acc(\drop(\alpha,D^*))=\alpha$. In particular, $\pi(\alpha)=\alpha$, and hence $\drop(\pi(\alpha),D^*)\s D'$.
So, $\pi^{-1}(\beta)=\beta$ for all $\beta\in \drop(\pi(\alpha),D^*)$, and hence $C_\alpha^*=\drop(\pi(\alpha),D^*)=\drop(\alpha,D^*)$.
\end{proof}

At this stage, note that the preceding claims are analogs of Claims \ref{322}, \ref{433}.
Now, continuing along the lines of the proof of the Theorem \ref{61}
(starting with Claim \ref{436}), it is clear how to utilize the established properties of $\langle C_\alpha^*\mid \alpha<\lambda^+\rangle$
to prove that $\sd'_{\lambda}(S)$ holds.
\end{proof}

\begin{lemma}\label{18} Suppose that $\lambda$ is a given uncountable cardinal.

Then $\sd'_\lambda(S)\Rightarrow \sd_\lambda(S)\Rightarrow \sc_\lambda(S)$ for every stationary $S\s\lambda^+$.
\end{lemma}
\begin{proof}
The proof of the second implication is a straight-forward variation of the proof of Lemma \ref{25} below,
so let us focus on establishing the first implication.

Let $\langle (C_\alpha,S_\alpha)\mid\alpha<\lambda^+\rangle$ be a witness to $\sd'_\lambda(S)$.
Let $\{ X_\gamma\mid \gamma<\lambda^+\}$ be some enumeration of $[\lambda\times\lambda^+]^{\le\lambda}$
such that each element of $[\lambda\times\lambda^+]^{\le\lambda}$ is enumerated cofinally often.
For all limit $\alpha<\lambda^+$ and $i<\lambda$, put
\begin{itemize}
\item $X^i_\alpha:=\{\eta\mid (i,\eta)\in X_\alpha\}$;
\item $S^i_\alpha:=\bigcup\{ X^i_\gamma\cap[\sup(C_\alpha\cap\gamma),\min(C_\alpha\bks\gamma+1))\mid \gamma\in S_\alpha\}.$
\end{itemize}
It easy to see that if $\beta\in\acc(C_\alpha)$, then $S^i_\beta=S^i_\alpha\cap\beta$.

\begin{claim} There exists some $i<\lambda$, such that for every club $D\s\lambda^+$,
and every $A\s\lambda^+$, for some $\alpha\in S$, we have:
\begin{enumerate}
\item[(a)] $C_\alpha\s D$;
\item[(b)] $S^i_\alpha=A\cap\alpha$;
\item[(c)] $\sup(\acc(C_\alpha))=\alpha$;
\item[(d)] $\otp(C_\alpha)=i$.
\end{enumerate}
\end{claim}
\begin{proof} Suppose not. Then for every $i<\lambda$, we may pick a pair $(D_i,A_i)$
witnessing that (a)$\wedge$(b)$\wedge$(c)$\wedge$(d) do not hold.
Put $H:=\bigcup_{i<\lambda}(\{i\}\times A_i)$.
Let $f:\lambda^+\rightarrow\lambda^+$ be some strictly increasing function satisfying $X_{f(\delta)}=H\cap(\lambda\times\delta+1)$
for all $\delta<\lambda^+$.
Put $A:=f``\lambda^+$ and $D:=\left\{ \beta\in\bigcap_{i<\lambda}D_i\mid f[\beta]\s\beta\right\}$.

Pick $\alpha\in S$ such that $C_\alpha\s D$, $S_\alpha=A\cap\alpha$ and $\sup(\acc(C_\alpha))=\alpha$.
Put $i:=\otp(C_\alpha)$. As $C_\alpha\s D\s D_i$, to meet a contradiction, it suffices to establish that $S^i_\alpha=A_i\cap\alpha$.
Evidently, $A_i\cap\alpha=\bigcup\{ X^i_{\gamma}\mid \gamma\in f[\alpha]\}$.
By $C_\alpha\s D$, we get that $f[\alpha]\s\alpha$, and hence $A_i\cap\alpha=\bigcup\{ X^i_{\gamma}\mid \gamma\in A\cap\alpha\}$.
That is, $A_i\cap\alpha=\bigcup\{ X^i_{\gamma}\mid \gamma\in S_\alpha\}$.
Recalling the definition of $S^i_\alpha$, we get that $S^i_\alpha\s A_i\cap\alpha$.
Finally, pick some $\delta\in A_i\cap\alpha$, and let us show that $\delta\in S^i_\alpha$.
Put $\gamma:=f(\delta)$, $\gamma^-:=\sup(C_\alpha\cap\gamma)$ and $\gamma^+:=\min(C_\alpha\bks\gamma+1)$.
Since $f$ is strictly increasing, we get that $\max\{\delta,\gamma^-\}\le \gamma<\gamma^+$.
If $\delta<\gamma^-$, then by $\gamma^-\in C_\alpha\s D$, we would get that $\gamma=f(\delta)<\gamma^-$.
So $\gamma^-\le\delta\le\gamma<\gamma^+$.

Since $f[\alpha]=S_\alpha$, we get that $\gamma\in S_\alpha$, and hence $\delta\in X^i_\gamma\cap[\gamma^-,\gamma^+)$.
\end{proof}
Let $i$ be given by the previous claim.
Put $S':=\{\alpha\in S\mid \otp(C_\alpha)=i\}$.
For all limit $\alpha<\lambda^+$, define $S_\alpha^*:=S_\alpha^i$, and
$$C^*_\alpha:=\begin{cases}
C_\alpha,&\otp(C_\alpha)\le i\\
\{\beta\in C_\alpha\mid\otp(C_\alpha\cap\beta)>i\},&\text{otherwise}
\end{cases}.$$

Then $\langle (C_\alpha^*,S^*_\alpha)\mid \alpha<\lambda^+\rangle$ and $S'$ witness the validity of $\sd_\lambda(S)$-sequence.
\end{proof}

\npg\section{The Ostaszewski square}\label{sec3}
In this section, we shall study the validity of $\sc_\lambda$.
Before doing so,  let us mention three variations of Definition \ref{def13} that  happens to be equivalent to the original one.
\begin{lemma} The existence of a $\square_\lambda$-sequence $\overrightarrow C=\langle C_\alpha\mid\alpha<\lambda^+\rangle$
satisfying Clause (3) of Definition \ref{def13} is equivalent to the existence of a $\square_\lambda$-sequence $\overrightarrow C$
that satisfies any of the following:
\begin{itemize}
\item[($3_{strong}$)] Suppose that $\langle A_i\mid i<\lambda\rangle$
is a sequence of unbounded subsets of $\lambda^+$.

Then for every limit $\theta<\lambda$, and every \underline{unbounded} $U\s\lambda^+$,
there exists some $\alpha<\lambda^+$ such that $\otp(C_\alpha)=\theta$, and for all $i<\theta$:
\begin{enumerate}
\item[(a)] $C_\alpha(i+1)\in A_i$;
\item[(b)] $C_\alpha(i)<\beta<C_\alpha(i+1)$ for some $\beta\in U$.
\end{enumerate}

\item[($3_{smooth}$)] Suppose that $\langle A_i\mid i<\lambda\rangle$
is a sequence of unbounded subsets of $\lambda^+$,
such that $A_i$ is closed, for all limit $i<\lambda$.

Then for every limit $\theta<\lambda$,
there exists some $\alpha<\lambda^+$ such that the  isomorphism $\pi_\alpha:\otp(C_\alpha)\rightarrow C_\alpha$
is a choice function in $\prod_{i<\theta}A_i$.

\item[($3_{succinct}$)] Suppose that $\langle A_i\mid i<\lambda\rangle$
is a sequence of unbounded subsets of $\lambda^+$.

Then for every limit $\theta<\lambda$,
there exists some $\alpha<\lambda^+$ such that $\otp(C_\alpha)=\theta$, and
$C_\alpha(i+1)\in A_i$ for all $i<\theta$.

\end{itemize}
\end{lemma}
\begin{proof} The implications $(3_{strong})\Rightarrow (3)\Rightarrow (3_{smooth})\Rightarrow (3_{succinct})$ are straight-forward,
hence we focus on showing that $(3_{succinct})\Rightarrow (3_{strong})$.

Suppose that $\langle C_\alpha\mid\alpha<\lambda^+\rangle$
witnesses the succinct version of $\sc_\lambda$.
Suppose that $\langle A_i\mid i<\lambda\rangle$
is a sequence of unbounded subsets of $\lambda^+$,
and that $U$ is some cofinal subset of $\lambda^+$.
Recursively define a function $f:\lambda^+\rightarrow\lambda^+$
insuring that:
\begin{itemize}
\item $f(0)>\min(U)$;
\item $f(\alpha)>\min(U\bks(\sup(f[\alpha])+1))$ for all nonzero $\alpha<\lambda^+$;
\item if $\alpha<\lambda^+$ and $i$ is the unique ordinal $<\lambda$,
such that $\alpha=\lambda\cdot\gamma+i$ for some $\gamma<\lambda^+$, then $f(\alpha)\in A_i$;
\end{itemize}
Of course, there is no problem in constructing such a function.

Now, for all $i<\lambda$, let $A_i':=\{ f(\lambda\cdot\gamma+i)\mid \gamma<\lambda^+\}$,
and note that $A_i'$ is a cofinal subset of $A_i$.
Finally, by the hypothesis, for every limit $\theta<\lambda$, we can find some $\alpha<\lambda^+$
such that $\otp(C_\alpha)=\theta$, and $C_\alpha(i+1)\in A_i'$ for all $i<\theta$.
Evidently, for all $i<\theta$, $C_\alpha(i)<\beta<C_\alpha(i+1)$ for some $\beta\in U$.
\end{proof}

Hereafter, we shall respect the original Definition \ref{def13}  when verifying the validity of $\sc_\lambda$.

\begin{lemma}\label{25} Suppose that $\lambda$ is an uncountable cardinal.

If $\sd_\lambda^{\Gamma}$ holds, for a subset $\Gamma\s\reg(\lambda^+)$ with $\sup(\Gamma)=\lambda$, then $\sc_\lambda$ is valid.
\end{lemma}
\begin{proof}
Let $\langle (C_\alpha,S_\alpha)\mid \alpha<\lambda^+\rangle$ witness the validity of $\sd_\lambda^{\Gamma}$.
Fix a bijection $\psi:\lambda\times\lambda^+\leftrightarrow\lambda^+$.
Fix a limit ordinal $\alpha<\lambda^+$.
Put $C_\alpha':=\acc(C_\alpha)$ in case that
$\sup(\acc(C_\alpha))=\alpha$, and let $C_\alpha'$ be some cofinal subset of $\alpha$ of order-type $\omega$, otherwise.
Next, for $\beta\in\nacc(C_\alpha')$, let:
\begin{itemize}
\item $X^\beta_\alpha:=\{\gamma\mid  \psi(\otp(\nacc(C_\alpha')\cap\beta),\gamma)\in  S_\beta\}$;
\item $Y^\beta_\alpha:=X^\beta_\alpha\cap(\min\left(C_\alpha\bks(\sup(C_\alpha'\cap\beta)+1)\right),\beta)$;
\item $\beta_\alpha:=\min(Y^\beta_\alpha\cup\{\beta\})$.
\end{itemize}

Evidently, $\sup(C_\alpha'\cap\beta)<\beta_\alpha\le\beta$.
For all limit $\alpha<\lambda^+$, put $$C_\alpha^\bullet:=\acc(C_\alpha')\cup\{\beta_\alpha\mid \beta\in\nacc(C_\alpha')\}.$$
\begin{claim}\label{251} $\langle C_\alpha^\bullet\mid \alpha<\lambda^+\rangle$ is a $\square_\lambda$-sequence.
\end{claim}
\begin{proof} It easy to see that for all $\alpha<\lambda^+$,
we have $\acc(C_\alpha^\bullet)=\acc(C_\alpha')$,
and $\otp(C_\alpha^\bullet)=\otp(C'_\alpha)\le\otp(C_\alpha)\le\lambda^+$.
Suppose that $\delta\in\acc(C_\alpha^\bullet)$. Then $C_\alpha'=\acc(C_\alpha)$, $\delta\in\acc(C_\alpha')$,
$C_\delta'=C_\alpha'\cap\delta$, and $C_\delta=C_\alpha\cap\delta$. In particular, $\beta_\alpha=\beta_\delta$ for all $\beta\in\nacc(C'_\delta)$,
and $C^\bullet_\beta=C^\bullet_\alpha\cap\beta$.
\end{proof}
\begin{claim} For every cofinal $D\s\lambda^+$,
every sequence $\langle A_i\mid i<\lambda\rangle$ of unbounded subsets of $\lambda^+$,
and every limit $\theta<\lambda$,
there exists some $\alpha<\lambda^+$ such that $\otp(C^\bullet_\alpha)=\theta$,
and if $\{ \alpha^\bullet_i\mid i<\theta\}$ denotes the increasing enumeration of $C^\bullet_\alpha$, then:
\begin{enumerate}
\item[(a)] $\alpha^\bullet_{i+1}\in A_{i}$ for all $i<\theta$;
\item[(b)] $(\alpha^\bullet_i,\alpha^\bullet_{i+1})\cap D\not=\emptyset$ for all $i<\theta$.
\end{enumerate}
\end{claim}
\begin{proof} Fix a limit $\theta<\lambda$. Put $A:=\psi[\bigcup_{i<\lambda}(\{i\}\times A_i)]$, and
$$D':=\left\{\alpha\in D\mid \psi[\lambda\times\alpha]=\alpha\right\}\cap \bigcap_{i<\lambda}\acc^+(A_i).$$

Since $\sup(\Gamma)=\lambda$, we may pick some $\alpha\in E^{\lambda^+}_{>\theta}$
such that $C_\alpha\s D'$, $S_\alpha=A\cap\alpha$ and $\sup(\acc(C_\alpha))=\alpha$.
In particular, $C_\alpha'=\acc(C_\alpha)$. Put $\epsilon=\otp(C^\bullet_\alpha)$.
Let $\{ \alpha^\bullet_i \mid i<\epsilon\}$
denote the increasing enumeration of $C^\bullet_\alpha$,
and let $\{ \alpha'_i \mid i<\epsilon\}$
denote the increasing enumeration of $C'_\alpha$.
Fix $i<\epsilon$. Then:
\begin{itemize}
\item $X^{\alpha'_{i+1}}_\alpha=\{\gamma\mid \psi(i,\gamma)\in  S_{\alpha'_{i+1}}\}$;
\item $Y^{\alpha'_{i+1}}_\alpha:=X^{\alpha'_{i+1}}_\alpha\cap(\min\left(C_\alpha\bks(\alpha'_{i}+1)\right),\alpha'_{i+1})$;
\item ${\alpha^\bullet_{i+1}}:=\min(Y^{\alpha'_{i+1}}_\alpha\cup\{{\alpha'_{i+1}}\})$.
\end{itemize}

So, $\alpha_i^\bullet\le \alpha_i'<\min(C_\alpha\bks \alpha_i'+1)<\alpha_{i+1}^\bullet\le\alpha_{i+1}$.

$\br$ By $C_\alpha\s D'\s D$, we infer that $(\alpha^\bullet_i,\alpha^\bullet_{i+1})\cap D\not=\emptyset$.

$\br$ By $\alpha_{i+1}'\in \acc(C_\alpha)$, we get that $\psi[\lambda\times\alpha_{i+1}']=\alpha_{i+1}'$
and $S_{\alpha'_{i+1}}=S_\alpha\cap\alpha'_{i+1}=A\cap\alpha'_{i+1}=\psi[\bigcup_{i<\lambda}\{i\}\times(A_i\cap\alpha'_{i+1})]$.
We also get that $\alpha_{i+1}'\in\acc^+(A_i)$, and hence $X^{\alpha'_{i+1}}_\alpha$ is a cofinal subset of $A\cap \alpha_{i+1}'$.
So $\alpha_{i+1}^\bullet<\alpha_{i+1}$, and $\alpha_{i+1}^\bullet\in A_i$.

Since $\theta$ is a limit ordinal, we get from Claim \ref{251} that $C^\bullet_{\alpha^\bullet_\theta}=\{\alpha^\bullet_i\mid i<\theta\}$
is a witness to the above properties (a) and (b).
\end{proof}
Thus, $\langle C_\alpha^\bullet\mid \alpha<\lambda^+\rangle$ exemplifies the validity of $\sc_\lambda$.
This completes the proof.
\end{proof}

\begin{thm}\label{301} Suppose that $\lambda$ is a successor cardinal, and $\lambda^{<\lambda}<\lambda^\lambda=\lambda^+$.

If $\square_\lambda$ holds, then so does $\sc_\lambda$.
\end{thm}
\begin{proof} Fix a $\square_\lambda$-sequence, $\langle C_\alpha\mid\alpha<\lambda^+\rangle$.
Put $$O:=\{ \gamma<\lambda \mid \gamma\text{ is indecomposable, and }\{\alpha\in E^{\lambda^+}_{<\lambda}\mid \otp(C_\alpha)=\gamma\}\text{ is stationary}\}.$$

\begin{claim} $O$ is a cofinal subset of $\lambda$.
\end{claim}
\begin{proof}
Note that since $\lambda$ is regular, we get that $\otp(C_\alpha)<\lambda$ for every $E^{\lambda^+}_{<\lambda}$.
Now, suppose that $\theta<\lambda$, and let us show that $O\bks\theta\not=\emptyset$.
Given a club $D\s\lambda^+$, let us pick some $\alpha\in E^{\lambda^+}_\lambda\cap\acc(D)$.
Since $\cf(\alpha)=\lambda$ is uncountable, we get that $D\cap\acc(C_\alpha)$ is a club of order-type $\lambda$,
so let us pick some $\beta\in D\cap\acc(C_\alpha)$ for which $\otp(C_\alpha\cap\beta)$ is indecomposable and $>\theta$.
Then $\beta\in D$ and $\otp(C_\beta)=\otp(C_\alpha\cap\beta)$ is as desired.
\end{proof}

Let $\Gamma:=\{\gamma_i\mid i<\lambda\}$ be the increasing enumeration of some club subset of $\lambda$,
with $\{\gamma_{i+1}\mid i<\lambda\}\s O$.
Denote $\gamma_{\lambda}=\lambda$.
For all limit $i\le\lambda$, denote $\Gamma_i:=\{\gamma_j\mid j<i\}$.
Next, for all limit $\epsilon\le\lambda$, we define the club $E_\epsilon\s\epsilon$ by letting:
$$E_\epsilon:=\begin{cases}
\Gamma_i,&\epsilon=\gamma_i\ \&\ i\text{ is limit},\\
\epsilon\bks\gamma_i+1,&\epsilon\in(\gamma_i,\gamma_{i+1}]\\
\end{cases}.$$

Now, given $C\s\lambda^+$, let $\epsilon_C:=\otp(C)$ and $\pi_C:\epsilon_C\rightarrow C$ denote the inverse collapse.
Next, for all limit $\alpha<\lambda^+$, put $C_\alpha':=\pi_{C_\alpha}``E_{\epsilon_{C_\alpha}}$.
Then $C_\alpha'$ is a subclub of $C_\alpha$ and in particular, $\acc(C_\alpha')\s\acc(C_\alpha)$.

\begin{claim}\label{312} For all $\alpha<\lambda^+$:
\begin{enumerate}
\item if $\otp(C_\alpha)=\gamma_{i+1}$ for some $i<\lambda$, then $\acc(C_\alpha')\s\acc(C_\alpha)$ and $\otp(C_\alpha')=\otp(C_\alpha)$.
\item if $\beta\in\acc(C_\alpha')$, then $C_\alpha'\cap\beta=C'_\beta$.
\end{enumerate}
\end{claim}
\begin{proof}
(1) If $\otp(C_\alpha)=\gamma_{i+1}$, then it follows from the definition of $C_\alpha'$,
that $C_\alpha'$ is a final segment of $C_\alpha$. Then $\acc(C_\alpha')\s\acc(C_\alpha)$.
Recalling that $\otp(C_\alpha)$ is indecomposable,
we also get that $\otp(C_\alpha)=\otp(C_\alpha')$.

(2) Suppose that $\alpha<\lambda^+$ and $\beta\in\acc(C_\alpha')$. Put $\epsilon:=\otp(C_\alpha)$,
and $\beta':=\pi^{-1}_{C_\alpha}(\beta)$. Then $\beta'\in\acc(E_\epsilon)$ and $\pi_{C_\beta}=\pi_{C_\alpha}\restriction\beta'$.

$\blacktriangleright$
If $\epsilon=\gamma_i$ for some limit $i\le\lambda$, then $E_\epsilon=\Gamma_i$, and hence,
$\beta'=\gamma_j$ for some limit $j<\lambda$. So
$\pi_{C_\beta}=\pi_{C_\alpha}\restriction\gamma_j$, and that $E_{\beta'}=\Gamma_j$.
In particular, $C_\beta'=\pi_{C_\beta}[\Gamma_j]=\pi_{C_\alpha}[\Gamma_i\cap\gamma_j]=C_\alpha'\cap\beta$.

$\blacktriangleright$
If $\epsilon\in(\gamma_i,\gamma_{i+1}]$ for some $i<\lambda$, then $\beta'\in(\gamma_i,\epsilon)$.
In particular, $\beta'\in(\gamma_i,\gamma_{i+1}]$, and hence $E_{\beta'}=\beta'\bks\gamma_i+1$.
So $C_\beta'=\pi_{C_\beta}[(\gamma_i,\beta')]=\pi_{C_\alpha}[(\gamma_i,\epsilon)\cap\beta']=C_\alpha'\cap\beta$.
\end{proof}

Given a club $D\s\lambda^+$, denote $$\drop(\alpha,D):=\{ \sup(D\cap\beta)\mid \beta\in C'_\alpha\bks \min(D)\}.$$

\begin{claim} For all $i<\lambda$, there exists a club $D_i\s\lambda^+$
such that for every club $E\s\lambda^+$, there exists some $\alpha\in\acc(D_i)$
with:
\begin{enumerate}
\item $\otp(C_\alpha)=\gamma_{i+1}$;
\item $\drop(\alpha,D_i)\s E$.
\end{enumerate}
\end{claim}
    \begin{proof} As in the proof of Claim \ref{351}.\end{proof}

For all $i<\lambda$, let $D_i\s\lambda^+$ be a club given by the previous claim.
Put $D^*:=\bigcap_{i<\lambda}D_i$,
and let $\pi:\lambda^+\rightarrow D^*$ denote  the inverse collapse.
For all limit $\alpha<\lambda^+$, let $$C_\alpha^*:=\pi^{-1}``\drop(\pi(\alpha),D^*).$$

\begin{claim}  $\langle C^*_\alpha\mid \alpha<\lambda^+\rangle$ is a $\square_\lambda$-sequence.
\end{claim}
\begin{proof} As in the proof of Claim \ref{232}.\end{proof}

For all $i<\lambda$, let
$$T_i:=\{ \alpha<\lambda^+\mid \acc(C_\alpha^*)=\acc(C_\alpha')\ \&\ \otp(C^*_\alpha)=\otp(C_\alpha)=\gamma_{i+1}\}.$$

\begin{claim}
For every club $D\s\lambda^+$ and every $i<\lambda$,
the set $$\{ \alpha\in T_i\mid C_\alpha^*\s D\}$$ is stationary.
\end{claim}
\begin{proof} Suppose that $D,E$ are club subsets of $\lambda^+$,
and that $i<\lambda$. Our aim is to find an $\alpha\in T_i$ such that  $C_\alpha^*\s D$, and $\alpha\in E$.

Put $E':=\{ \beta\in D\cap E\cap\acc(D^*)\mid \pi(\beta)=\beta\}$.
By the choice of $D_i$,
we may pick some $\alpha\in \acc(D_i)$ such that $\otp(C_\alpha)=\gamma_{i+1}$ and $\drop(\alpha,D_i)\s E'$.
By Claim \ref{312}, we infer that $\otp(C_\alpha')=\gamma_{i+1}$.

For every $\beta\in C_\alpha'$, we have $\sup(D_i\cap\beta)\in E'\s D^*\s D_i$.
So $\sup(D^*\cap\beta)=\sup(D_i\cap\beta)$ for all $\beta\in C_\alpha'$,
and hence $\drop(\alpha,D^*)=\drop(\alpha,D_i)$.
Since $\drop(\alpha,D^*)\s\{\beta<\lambda^+\mid \pi(\beta)=\beta\}$, we get that $C_\alpha^*=\drop(\alpha,D_i)$.

For every $\beta\in\acc(C_\alpha')$, we have $\sup(D_i\cap\beta)\in E'\s\acc(D^*)\s \acc(D_i)$.
So $\sup(D_i\cap\beta)\in\acc(D_i)$ for all $\beta\in\acc(C_\alpha')$, and hence $\acc(C_\alpha^*)=\acc(C_\alpha')$.
In particular, $\otp(C_\alpha^*)=\otp(C_\alpha')=\gamma_{i+1}$, and hence $\alpha\in T_i$.
\end{proof}

For all $i<\lambda$, let $S_i:=\bigcup\{\acc(C_\alpha^*)\cup\{\alpha\}\mid \alpha\in T_i\}$.

\begin{claim} For all  $i<j<\lambda$, $S_i\cap S_j=\emptyset$.

In particular, $|\{ i<\lambda\mid (\acc(C_\alpha^*)\cup\{\alpha\})\cap S_i\not=\emptyset\}|\le 1$ for all limit $\alpha<\lambda^+$.
\end{claim}
\begin{proof} Suppose not. Pick $\alpha_i\in T_i$, $\alpha_j\in T_j$ such that
$(\acc(C_{\alpha_i}')\cup\{\alpha_{i}\})\cap (\acc(C_{\alpha_j}')\cup\{\alpha_{j}\})\not=\emptyset$,
and let $\beta$ be some element of this nonempty set.
By $\alpha_i\in T_i$, and Claim \ref{312},
we get that $\beta\in\acc(C_{\alpha_i})\cup\{\alpha_i\}$.
So $\otp(C_\beta)\le\otp(C_{\alpha_i})=\gamma_{i+1}<\gamma_{j+1}=\otp(C_{\alpha_j})$.
In particular, $\beta\not=\alpha_j$. Namely, $\beta\in\acc(C_{\alpha_j})$.
As $C_\beta\s C_{\alpha_j}$ and $\otp(C_\beta)\le\gamma_{i+1}$, we get that $\pi_{C_{\alpha_j}}^{-1}(\beta)\le\gamma_{i+1}$.
By $C_{\alpha_j}'=\pi_{C_{\alpha_j}}``E_{\epsilon_{C_{\alpha_j}}}$ and $\beta\in\acc(C'_{\alpha_j})$, this means that $\gamma_{i+1}\cap E_{\epsilon_{C_{\alpha_j}}}\not=\emptyset$.
However, $\alpha_j\in T_j$, and hence $\epsilon_{C_{\alpha_j}}=\gamma_{j+1}$,
so $E_{\epsilon_{C_{\alpha_j}}}=\Omega_{j+1}$, and $\min(E_{\epsilon_{C_{\alpha_j}}})=\gamma_{j}+1>\gamma_j\ge\gamma_{i+1}$.
This is a contradiction.

Now, as in the proof of Claim \ref{435}, we get that
$|\{ i<\lambda\mid (\acc(C_\alpha^*)\cup\{\alpha\})\cap S_i\not=\emptyset\}|\le 1$ for all limit $\alpha<\lambda^+$.
\end{proof}

By $\lambda^{<\lambda}=\lambda$ and the Engelking-Kar{\l}lowicz theorem \cite{ek},
let us fix a sequence of functions
$\langle f_j:\lambda^+\rightarrow\lambda\mid j<\lambda\rangle$ with the property that for every $X\in[\lambda^+]^{<\lambda}$,
and every function  $g:X\rightarrow\lambda$, there exists some $j<\lambda$ with $g\s f_j$.

By $\lambda^\lambda=\lambda^+$, let $\{ X_\beta\mid\beta<\lambda^+\}$ be some enumeration of $[\lambda\times\lambda\times\lambda^+\times\lambda^+]^{\le\lambda}$.
Let $\langle \psi_\delta:\lambda\rightarrow\delta\mid\delta<\lambda^+\rangle$ be some sequence of surjections.
Denote $\varphi_\beta(\gamma):=X_{\psi_\beta(\gamma)}$.
Next, given $j<\lambda$, $\alpha<\lambda^+$, and $\beta\in\nacc(C^*_\alpha)$,
let
\begin{itemize}
\item $X^{\beta}_{\alpha,j}=\{ \eta\mid(j,\otp(\nacc(C_\alpha^*)\cap\beta),\eta,\nu)\in \varphi_{\min(C^*_\alpha\bks\beta+1)}(f_j(\beta)),\nu<\eta\}$;
\item $Y^{\beta}_{\alpha,j}:=X^{\beta}_{\alpha,j}\cap(\sup(C^*_\alpha\cap\beta),\beta]$;
\item $j(\alpha,\beta):=\min(Y^{\beta}_{\alpha,j}\cup\{{\beta}\})$.
\end{itemize}

Note that $\sup(C_\alpha^*\cap\beta)<j(\alpha,\beta)\le\beta$.
Next, for all limit $\alpha<\lambda^+$, and $j<\lambda$, let:
$$C_\alpha^j:=\acc(C_\alpha^*)\cup\{ j(\alpha,\beta)\mid \beta\in\nacc(C_\alpha^*)\}.$$
\begin{claim} Fix $i<\lambda$. Then there exists some $j<\lambda$ with the following property.

for every club $D\s\lambda^+$,
every sequence
$\overrightarrow A=\langle A_\iota\mid \iota<\gamma_{i+1}\rangle$
of unbounded subsets of $\lambda^+$,
there exists some $\alpha\in T_i$ such that:
\begin{enumerate}
\item the inverse collapse of $\nacc(C^j_\alpha)$ belongs to $\prod\overrightarrow{A}$;
\item the open interval $(\sup(C^j_\alpha\cap\beta),\beta)$ meets $D$, for all $\beta\in\nacc(C^j_\alpha)$.
\end{enumerate}
\end{claim}
\begin{proof} Suppose not. Then for every $j<\lambda$, let us pick a club $D_j\s\lambda^+$,
and a sequence $\langle A^j_\iota\mid \iota<\gamma_{i+1}\rangle$ witnessing the failure of the above.
Put $$X:=\{ (j,\iota,\eta,\nu)\mid j<\lambda, \iota<\gamma_{i+1}, \eta\in A^j_\iota,\nu\in D_j\}. $$
Define $f:\lambda^+\rightarrow\lambda^+$ by letting for all $\epsilon<\lambda^+$:
$$f(\epsilon):=\min\{\varepsilon<\lambda^+\mid X\cap(\lambda\times\lambda\times\epsilon\times\epsilon)=X_{\varepsilon}\}.$$
Put $$D:=\{ \delta\in\bigcap_{j<\lambda}\acc(D_j) \mid f[\delta]\s\delta\}\cap\bigcap\{\acc^+(A^j_\gamma)\mid (j,\iota)\in\lambda\times\gamma_{i+1}\}.$$

Pick $\alpha\in T_i$ such that $C_\alpha^*\s D$.
Since $f[\delta]\s\delta$ for all $\delta\in C_\alpha^*$,
we get that $f(\beta)<\min(C_\alpha^*\bks\beta+1)$ for all $\beta\in C_\alpha^*$.
Hence, we may define $g:C_\alpha^*\rightarrow\lambda$ by letting for all $\beta\in C_\alpha^*$
$$g(\beta):=\min\{\varrho<\lambda\mid \psi_{\min(C^*_\alpha\bks\beta+1)}(\varrho)=f(\beta)\}.$$

Pick $j<\lambda$
such that $g\s f_j$.
Next, fix $\beta'\in\nacc(C^j_\alpha)$, and denote $\iota=\otp(\nacc(C_\alpha^j)\cap\beta')$.
Then $\beta'=j(\alpha,\beta)$ for $\beta:=(C_\alpha^*\bks\beta')$,
and $\otp(\nacc(C_\alpha^*)\cap\beta)=\iota$.
Also, for every $\beta\in C_\alpha^*$, we have:
$$X\cap(\lambda\times\lambda\times\beta\times\beta)=X_{f(\beta)}=X_{\psi_{\min(C^*_\alpha\bks\beta+1)}(g(\beta))}=\varphi_{\min(C^*_\alpha\bks\beta+1)}(g(\beta)),$$
and hence
$$X\cap(\lambda\times\lambda\times\beta\times\beta)=\varphi_{\min(C^*_\alpha\bks\beta+1)}(f_j(\beta)).$$

As $\beta\in C_\alpha^*\s D$, we get that $\beta\in\acc(D_j)\cap\acc(A^\iota_j)$.
In particular, there exists $\nu<\eta<\beta$ with $\nu>\sup(C_\alpha^*\cap\beta)$ such that $\nu\in D_j$
and $\eta\in A^\iota_j$. In other words, the set $Y^{\beta}_{\alpha,j}$ is non-empty,
 $\beta'=j(\alpha,\beta)\in A^\iota_j$, and $\sup(C_\alpha^j\cap\beta')\le \sup(C^*_\alpha\cap\beta)<\nu<j(\alpha,\beta)=<\beta$ for some $\nu\in D_j$.
So, $\beta'\in A^{\otp(\nacc(C_\alpha^j)\cap\beta)}_j$ and $\sup(C_\alpha^j\cap\beta',\beta')\cap D_j\not=\emptyset$.
This is a contradiction to the choice of $D_j$ and $\langle A^\iota_j\mid \iota<\gamma_{i+1}\rangle$.\end{proof}

For all $i<\lambda$, let $j_i<\lambda$ be given by the previous claim.
Fix a limit ordinal $\alpha<\lambda^+$.
As we have already noticed, for all $i<\lambda$,
$\acc(C^{j_i}_\alpha)=\acc(C^*_\alpha)$ and $C^{j_i}_\alpha\cap(\sup(C^*_\alpha\cap,\beta),\beta]$ is a singleton for all $\beta\in\nacc(C^*_\alpha)$.
Consequently, the following defines a club subset of $\alpha$, with the same order-type as $C^*_\alpha$.
$$C_\alpha^\bullet:=\begin{cases}C_\alpha^*,&(\acc(C_\alpha^*)\cup\{\alpha\})\cap\bigcup_{i<\lambda}S_i=\emptyset\\
C_\alpha^{j_i},&(\acc(C_\alpha^*)\cup\{\alpha\})\cap S_i\not=\emptyset
\end{cases}$$

\begin{claim} $\overrightarrow C:=\langle C_\alpha^\bullet\mid \alpha<\lambda^+\rangle$ is a $\sc_\lambda$-sequence.
\end{claim}
\begin{proof} We commence with verifying that $\overrightarrow C$ is coherent.
Suppose that $\delta\in\acc(C^\bullet_\alpha)$. Since $\acc(C^\bullet_\alpha)=\acc(C^*_\alpha)$,
we get that $C^*_\delta=C^*_\alpha\cap\delta$. We now distinguish three cases.

$\br$ Suppose that $(\acc(C_\delta^*)\cup\{\delta\})\cap S_i\not=\emptyset$
for some $i<\lambda$. Then $(\acc(C_\alpha^*)\cup\{\alpha\})\cap S_i\not=\emptyset$
and $C_\alpha^\bullet=C^{j_i}_\alpha, C_\delta^\bullet=C^{j_i}_\delta$ for this unique $i<\lambda$.
A quick look at the its definitions, shows that in this case $X_{\alpha,j_i}^\beta=X_{\beta,j_i}^\beta$,
and $Y_{\alpha,j_i}^\beta=Y_{\beta,j_i}^\beta$ for all $\beta\in\nacc(C^*_\delta)$.
So $\nacc(C^{j_i}_\alpha)\cap\delta=\nacc(C^{j_i}_\delta)$, and hence $C^\bullet_\alpha\cap\delta=C^\bullet_\delta$.

$\br$ Suppose that $C^\bullet_\alpha=C^*_\alpha$ and $C^\bullet_\delta=C^*_\delta$.
Then $C^\bullet_\alpha\cap\delta=C^\bullet_\delta$, and we are done.

$\br$ Neither of the above cases. This means that $(\acc(C_\delta^*)\cup\{\delta\})\cap\bigcup_{i<\lambda}S_i=\emptyset$,
while, $\gamma\in(\acc(C^*_\alpha)\cup\{\alpha\})\cap S_i$ for some $i<\lambda$.
Note that by definition of $S_i$, we get that $\acc(C^*_\gamma)\cup\{\gamma\}\s S_i$. Now, if $\gamma>\delta$,
then $\delta\in\acc(C^*_\gamma)$, contradicting the fact that $(\acc(C_\delta^*)\cup\{\delta\})\cap S_i=\emptyset$.
If $\gamma<\delta$, then $\gamma\in (\acc(C^*_\alpha)\cup\{\alpha\})\cap \delta=C^*_\delta$,
contradicting the very same fact. Altogether, this case does not exist.

Finally, let us verify the guessing property.
Fix a club $D\s\lambda^+$,
a sequence of unbounded subsets of $\lambda^+$, $\overrightarrow A=\langle A_i\mid i<\lambda\rangle$,
and some limit ordinal $\theta<\lambda$. Fix a large enough $i^*<\lambda$ such that $\gamma_{i^*+1}>\theta$.
By the choice of $j_{i^*}$, we may pick some $\alpha\in T_{i^*}$ for which
\begin{enumerate}
\item the inverse collapse of $\nacc(C^{j_{i^*}}_\alpha)$ belongs to $\prod(\overrightarrow{A}\restriction \gamma_{i^*+1})$;
\item the open interval $(\sup(C^{j_{i^*}}_\alpha\cap\beta),\beta)$ meets $D$, for all $\beta\in\nacc(C^{j_{i^*}}_\alpha)$.
\end{enumerate}
Let $\alpha^\bullet$ be the unique element of $C^{j_{i^*}}_\alpha$ with $\otp(C^{j_{i^*}}_\alpha\cap\alpha^\bullet)=\theta$.
Then $C^\bullet_{\alpha^\bullet}= C^{j_{i^*}}_\alpha\cap\alpha^\bullet$,
and if $\{ \alpha^\bullet_i\mid i<\theta\}$ denotes the increasing enumeration of $C_\alpha$, then
$\alpha_{i+1}\in A_i$ and $(\alpha_i,\alpha_{i+1})\cap D\not=\emptyset$ for all $i<\theta$.
\end{proof}
We have demonstrated the validity of $\sc_\lambda$, and hence the proof is complete.
\end{proof}

We conclude this section with an observation concerning the effect of $\sc_\lambda$ on cardinal arithmetic.
\begin{prop}\label{p34} $\sc_\lambda$ entails $\ch_\lambda$ for every singular cardinal $\lambda$.
\end{prop}
\begin{proof} Let $\overrightarrow C=\langle C_\alpha\mid \alpha<\lambda^+\rangle$ witness $\sc_\lambda$.
Let $\{ H_i\mid i<\lambda\}$ be a partition of $\lambda^+$ into $\lambda$-many mutually disjoint sets of size $\lambda^+$.
For all $\alpha<\lambda^+$, let $$X_\alpha=\{ i<\lambda\mid H_i\cap\nacc(C_\beta)\not=\emptyset\text{ for some }\beta\in \nacc(C_\alpha)\}.$$
\begin{claim} $\mathcal P(\lambda)=\{ X_\alpha\mid \alpha<\lambda^+\}$.
\end{claim}
\begin{proof} Suppose that $X\in\mathcal P(\lambda)$, and let us find some $\alpha<\lambda^+$ such that $X_\alpha=X$.
Fix a surjection $\pi:\lambda\rightarrow X$. Consider the sequence $\langle A_i\mid i<\lambda\rangle:=\langle H_{\pi(i)}\mid i<\lambda\rangle$.
Let $\langle \lambda_j\mid j<\lambda^+\rangle$ denote a strictly increasing sequence of regular cardinals converging to $\lambda$.
By the choice of $\overrightarrow C$, the set $$B_j:=\{ \beta<\lambda^+\mid \otp(C_\beta)=\lambda_j\ \&\ i<\lambda_j\Rightarrow C_\beta(i+1)\in A_i\}$$
is stationary for all $j<\cf(\lambda)$. Thus, appealing again to the defining properties of $\overrightarrow C$,
one can find some $\alpha<\lambda^+$ such that $\otp(C_\alpha)=\cf(\lambda)$
and $C_\alpha(j+1)\in B_j$ for all $j<\cf(\lambda)$. We claim that $X_\alpha=X$.

$\blacktriangleright$  Suppose that $\delta\in X$, and let us show that $\delta\in X_\alpha$.

Fix $i<\lambda$ such that $\pi(i)=\delta$,
and a large enough $j<\cf(\lambda)$ such that $i<\lambda_j$. Put $\beta:=C_\alpha(j+1)$.
Then $\beta\in B_j$, and in particular $C_\beta(i+1)\in A_i=H_{\pi(i)}$. So, we have found some
$\beta\in\nacc( C_\alpha)$ such that $\nacc(C_\beta)\cap H_{\pi(i)}\not=\emptyset$, and hence $\delta=\pi(i)\in X_\alpha$.

$\blacktriangleright$  Suppose that $\delta\in X_\alpha\bks X$, and let us meet a contradiction.

Fix $\beta\in\nacc(C_\alpha)$ such that $H_\delta\cap\nacc(C_\beta)\not=\emptyset$.
As $\beta\in\nacc(C_\alpha)$, we may fix some $j<\cf(\lambda)$ such that $\beta=C_\alpha(j+1)$.
So, $\beta\in B_j$, and hence $\otp(C_\beta)=\lambda_j$ and we may find some $i<\lambda_j$ such that $C_\beta(i+1)\in H_\delta$.
It follows that $A_i\cap H_\delta\not=\emptyset$, that is, $H_{\pi(i)}\cap H_\delta\not=\emptyset$,
which must mean that $\pi(i)=\delta$. In particular, $\delta\in\im(\pi)=X$. This is a contradiction.
\end{proof}
\end{proof}

\npg\section{An homogenous Souslin tree}\label{sec4}
In \cite[$\S6.1$]{jensen}, Jensen constructs a $\lambda^+$-Souslin tree for any cardinal $\lambda$,
provided that $V=L$. In \cite[$\S4$]{devlin}, Devlin extracts the actual hypotheses used in Jensen's construction, as follows.
\begin{thm}[Jensen, \cite{devlin}]\label{41} Assume $\gch$. Let $\lambda$ be an uncountable cardinal for which $\square_\lambda$ holds.
Then there exists a $\lambda^+$-Souslin tree.
\end{thm}

In \cite{velickovic}, Veli\v{c}kovi\'{c} presents a construction of a different nature. More specifically,  he constructs an $\aleph_2$-Souslin tree
which is strongly homogenous, assuming $\sd_{\aleph_1}(E^{\aleph_2}_{\aleph_1})$.\footnote{
For the definition, as well as characterizations of \emph{strong homogeneity}, see \cite{gido}.}
This construction generalizes to yield this kind of $\lambda^+$-tree from an analogous hypothesis,
for every regular uncountable  $\lambda$.
Moreover, the regularity of $\lambda$ appears to be essential, as the argument uses the hypothesis
concerning the validity of $\sd_\lambda(S)$ for $S=E^{\lambda^+}_\lambda$ in
a way that will not work for $S\s E^{\lambda^+}_{<\lambda}$.

As $E^{\lambda^+}_\lambda$ happens to be an empty set
whenever $\lambda$ is a singular cardinal, the author wondered for  quite a while, how could this type of construction may be carried
in the absence of the notion of ``maximal cofinality''.
It turns out that the missing observation is that Veli\v{c}kovi\'{c}'s construction
may be rendered as an application of a particular form of $\sc_\lambda$, which we introduce in Definition \ref{def41} below,
and which makes sense also for a singular $\lambda$:

\begin{thm}[Veli\v{c}kovi\'{c}]\label{t41} If $\ch_\lambda+\sc_{\lambda,1}^{\{\lambda\},1}$ holds for an uncountable cardinal $\lambda$,
then there exists a strongly  homogenous $\lambda^+$-Souslin tree.
\end{thm}
\begin{defn}\label{def41} $\sc^{\Gamma,\mu}_{\lambda,\kappa}$ asserts the existence of a sequence $\langle \mathcal C_\alpha\mid\alpha<\lambda^+\rangle$ such that:
\begin{enumerate}
\item for all limit $\alpha<\lambda^+$, $\mathcal C_\alpha$ is a nonempty collection of club subsets of $\alpha$;
\item if $C\in\mathcal C_\alpha$, then $\otp(C)\le\lambda$, and if $\beta\in\acc(C)$,
then $C\cap\beta\in\mathcal C_\beta$;
\item $|\mathcal C_\alpha|\le\kappa$ for all $\alpha<\lambda^+$;
\item  for every cofinal $A\s\lambda^+$,
and every limit $\theta\in\Gamma$,
there exists some $\alpha<\lambda^+$ for which all of the following holds:
\begin{enumerate}
\item $|\mathcal C_\alpha|\le \mu$;
\item $\nacc(C)\s A$ for all $C\in\mathcal C_\alpha$.
\item $\theta$ divides $\otp(C)$ for all $C\in\mathcal C_\alpha$;
\end{enumerate}
\end{enumerate}
\end{defn}

In Proposition \ref{43} below, an homogenous $\lambda^+$-Souslin tree
is constructed from $\ch_\lambda+\sc_{\lambda,\lambda}^{\{\lambda\},\lambda}$. That is,
we weaken the conclusion of Theorem \ref{t41} from ``strongly homogenous'' to  ``homogenous'',
while reducing the hypothesis from $\sc_{\lambda,1}^{\{\lambda\},1}$ to $\sc_{\lambda,\lambda}^{\{\lambda\},\lambda}$.
This weakening allows the constructed tree to enjoy an optimal degree of completeness.
More importantly, the value of the reduction is witnessed by the results of the previous sections when combined with the following lemma.

\begin{lemma}\label{lemma42} If $\lambda=\lambda^{<\cf(\lambda)}$ is a singular cardinal,
then $\sc_\lambda$ entails $\sc_{\lambda,\lambda}^{\{\lambda\},1}$.
\end{lemma}
\begin{proof}
Let $\overrightarrow C=\langle C_\alpha\mid \alpha<\lambda^+\rangle$ witness $\sc_\lambda$.
For simplicity, assume that $C_{\alpha+1}=\emptyset$ for all $\alpha<\lambda^+$.
For every $\delta<\lambda^+$, denote $C_\delta':=\{\beta\in C_\delta\mid \otp(C_\delta\cap\beta)>\cf(\lambda)\}$.
For every subset $c\s\lambda^+$, denote $$\varphi(c):=c\cup \{C_\delta'\bks\sup(c\cap\delta)\mid \delta\in\nacc(c)\ \&\ \cf(\delta)>\cf(\lambda)\}.$$
Put $$S:=\{\alpha<\lambda^+\mid \otp(C_\alpha)\le\cf(\lambda)\}.$$

\begin{claim}\label{421} For $c\s\lambda^+$, we have:
\begin{enumerate}
\item if $c$ is closed, then so does $\varphi(c)$;
\item if $\otp(c)\le\cf(\lambda)$, then $\otp(\varphi(c))\le\lambda$;
\item if $\beta\in\acc(\varphi(c))\bks\acc(c)$, then $\beta\not\in S$.
\end{enumerate}
\end{claim}
\begin{proof}
(2) Evidently, if $\delta<\lambda^+$ and $\cf(\delta)>\cf(\lambda)$,
then $\otp(C_\delta)<\lambda$. It follows that if $\otp(c)\le\cf(\lambda)$,
then $\otp(\varphi(c)\cap\alpha)<\lambda$ for all $\alpha<\sup(c)$, and hence $\otp(\varphi(c))\le\lambda$.

 (3) Suppose that $\beta\in\acc(\varphi(c))\bks\acc(c)$.
If $\beta\in c$, then $\beta\in\acc(\varphi(c))\cap\nacc(c)$,
which must mean that $\cf(\beta)>\cf(\lambda)$. As $S\s E^{\lambda^+}_{\le\cf(\lambda)}$,
we infer that $\beta\not\in S$.
Next, suppose that $\beta\not\in c$. Then there exists some $\delta\in\nacc(c)\cap E^{\lambda^+}_{>\cf(\lambda)}$
such that $\beta\in\acc(C_\delta')$.  So $\beta\in\acc(C_\delta)$,
and $\otp(C_\beta)=\otp(C_\delta\cap\beta)$. Recalling that $\beta\in C_\delta'$, we get that $\otp(C_\delta\cap\beta) >\cf(\lambda)$,
so $\otp(C_\beta)>\cf(\lambda)$ and hence $\beta\not\in S$.
\end{proof}

Next, for all limit $\alpha<\lambda^+$, we define the following:
\begin{itemize}
\item If $\alpha\in S$, let $\mathcal C_\alpha:=\{ \varphi(C_\alpha)\}$;
\item If $\alpha\not\in S$, let
$\mathcal C_\alpha:=\{ \varphi\left(c \cup\{\alpha\}\right)\mid c\in[\alpha]^{<\cf(\lambda)}\text{ is closed}\}$.
\end{itemize}

By the preceding claim, for every limit $\alpha<\lambda^+$, every $C\in\mathcal C_\alpha$
is a club subset of $\alpha$ of order-type $\le\lambda$.
As $\lambda^{<\cf(\lambda)}=\lambda$, we also have that $|\mathcal C_\alpha|\le\lambda$. 
\begin{claim} If $C\in\mathcal C_\alpha$ and $\beta\in\acc(C)$, then $C\cap\beta\in\mathcal C_\beta$.
\end{claim}
\begin{proof} Suppose that $C,\alpha,\beta$ are as above.

$\blacktriangleright$ If $\alpha\in S$ and $\beta\in\acc(C_\alpha)$, then $C\cap\beta=\varphi(C_\alpha)\cap\beta=\varphi(C_\alpha\cap\beta)=\varphi(C_\beta)$,
$\otp(C_\beta)\le\otp(C_\alpha)\le\cf(\lambda)$, and hence $\mathcal C_\beta=\{C\cap\beta\}$.

$\blacktriangleright$ If $\alpha\in S$ and $\beta\not\in\acc(C_\alpha)$, put $c:=C_\alpha\cap\beta$.
Then $c\in[\beta]^{<\cf(\lambda)}$, and by Claim \ref{421}, $\beta\not\in S$.
If $\beta\in\nacc(C_\alpha)$, then it is clear that $C\cap\beta=\varphi(c\cup\{\beta\})$.
If $\beta\not\in\nacc(C_\alpha)$, then $\beta\in\acc(C_\delta')$ for $\delta:=\min(C_\alpha\bks\beta)$,
and since $C_\beta'=C_\delta'\cap\beta$, we get once again that $C\cap\beta=\varphi(c\cup\{\beta\})$.
Recalling that $\beta\not\in S$, we conclude that $C\cap\beta\in\mathcal C_\beta$.

$\blacktriangleright$ If $\alpha\not\in S$,
then $\beta\in\acc(\varphi(c\cup\{\alpha\}))$ for some closed $c\in[\alpha]^{<\cf(\lambda)}$.
Put $\delta:=\min((c\cup\{\alpha\})\bks\beta)$. Then  $\beta\in\acc(C_\delta')$,
and similarly to the preceding case, $\beta\not\in S$, and $C\cap\beta=\varphi((c\cap\beta)\cup\{\beta\})\in\mathcal C_\beta$.
\end{proof}

\begin{claim} For every cofinal $A\s\lambda^+$, there exists some $\alpha\in S$
such that $\otp(\varphi(C_\alpha))=\lambda$ and $\nacc(\varphi(C_\alpha))\s A$.
\end{claim}
\begin{proof}
Suppose that $A\s\lambda^+$ is as above. Let $\langle \lambda_i\mid i<\cf(\lambda)\rangle$
be a strictly increasing sequence of cardinals converging to $\lambda$.
For all $i<\cf(\lambda)$, put $$A_i:=\{\delta<\lambda^+\mid\otp(C_\delta)=\lambda_i\ \&\ \nacc(C_\delta)\s A\}.$$
Since $\overrightarrow C$ is a $\sc_\lambda$-sequence, $\langle A_i\mid i<\cf(\lambda)\rangle$
happens to be a sequence of unbounded subsets of $\lambda^+$.
Since $\overrightarrow C$ is a $\sc_\lambda$-sequence, we may find some $\alpha<\lambda^+$
such that $\otp(C_\alpha)=\cf(\lambda)$, and $C_\alpha(i+1)\in A_i$ for all $i<\cf(\lambda)$.

It follows that $\varphi(C_\alpha)$ is of order-type $\lambda$
and $\nacc(\varphi(C_\alpha))\s A$.
Finally, note that by $\otp(C_\alpha)=\cf(\lambda)$, we get that  $\alpha\in S$.
\end{proof}
By the previous claims, and recalling that $\mathcal C_\alpha=\{\varphi(C_\alpha)\}$ for all $\alpha\in S$,
we conclude that $\langle \mathcal C_\alpha\mid \alpha<\lambda^+\rangle$
witnesses $\sc^{\{\lambda\},1}_{\lambda,\lambda}$.
\end{proof}

\begin{prop}\label{43}
If $\sc_{\lambda,\lambda}^{\{\lambda\},\lambda}+\ch_\lambda$ holds for an uncountable cardinal $\lambda$,
then there exists an homogenous $\lambda^+$-Souslin tree. Moreover, there exists one which is $\kappa$-complete,
for $\kappa:=\min\{\theta\mid \lambda^\theta\not=\lambda\}$.
\end{prop}
\begin{proof}
We follow closely Veli\v{c}kovi\'{c} construction from \cite{velickovic}.

The resulting tree $\T$ would be a subtree of $\langle {}^{<\lambda^+}2,\s\rangle$,
and the $\alpha_{th}$ level of the tree, which we shall denote by $T_\alpha$, will be a subset of ${}^\alpha2$.
We also denote $\T\restriction\alpha:=\bigcup_{\beta<\alpha}T_\beta$.
For sequences $s,t\in{}^{<\lambda^+}2$ with $|s|<|t|$,
let $s*t:=s\cup (t\bks(\dom(s)\times 2))$  denote the sequence that begins as $s$ and continues as $t$.
Also, let $t^\frown i:=t\cup\{(\dom(t),i)\}$ denote the concatenation of the sequence $t$ with the sequence $\langle i\rangle$.

Next, let us fix the following objects:
\begin{itemize}
\item  a  $\sc_{\lambda,\lambda}^{\{\lambda\},\lambda}$-sequence,
$\langle \mathcal C_\alpha\mid \alpha<\lambda^+\rangle$, that exists by hypothesis;
\item an enumeration $\{C_\alpha^j\mid j<\lambda\}$ of $\mathcal C_\alpha$, for each limit $\alpha<\lambda^+$;
\item a $\diamondsuit({\lambda^+})$-sequence, $\langle S_\gamma\mid\gamma<\lambda^+\rangle$, that exists by $\ch_\lambda$ and \cite{sh922};
\item a bijection $\varphi:{}^{<\lambda^+}2\leftrightarrow\lambda^+$, that exists by $\ch_\lambda$;
\item a surjection $f:\lambda\rightarrow\lambda\times\lambda$ such that if $f(i)=\langle k,\nu\rangle$,
then $k \le i$;
\item a well ordering $\preceq$ of ${}^{<\lambda^+}2$.
\end{itemize}

We now turn to the construction. For every $\alpha<\lambda^+$, we shall construct the following objects:
\begin{itemize}
\item $T_\alpha\s{}^{\alpha}2$;
\item a surjection $\psi_\alpha:\lambda\rightarrow\T\restriction\alpha+1$;
\item a collection $\{t_\alpha^j\mid j<\lambda\}\s T_\alpha$ for all nonzero limit $\alpha<\lambda^+$.
\end{itemize}

\underline{Construction Base:}
 Let $T_0=\{\emptyset\}$.
Let $\psi_\alpha:\lambda\rightarrow T_0$ be the constant function.

\underline{Successor Stage:}  If $T_\alpha$ is defined,
then $T_{\alpha+1}:=\{ \sigma^\frown0,\sigma^\frown1\mid \sigma\in T_\alpha\}$
and $\psi_\alpha:\lambda\rightarrow T\restriction\alpha+1$ would be some arbitrary, fixed  surjection.

\underline{Limit Stage:}
Suppose that $\alpha$ is a limit ordinal, and that $T_\beta,\psi_\beta$ have already been defined for all $\beta<\alpha$.
We consider two cases.

$\blacktriangleright$ (Case $\aleph$) Suppose that $\otp(C_\alpha^j)<\lambda$ for all $j<\lambda$,
and that the number of cofinal branches in $(\T\restriction\alpha)$ is $\le\lambda$.
In this case, let $\{ t^j_\alpha\mid j<\lambda\}$ be some fixed enumeration of all these branches.
Then, put $T_\alpha:=\{t^j_\alpha\mid j<\lambda\}$, and note that
$$T_\alpha=\{ s*t^j_\alpha\mid s\in(\T\restriction\alpha), j<\lambda\}.$$

$\blacktriangleright$ (Case $\beth$) Suppose that the above does not apply.
Fix $j<\lambda$. Let $\{ \alpha^j_i\mid i<\otp(C_\alpha^j)\}$ be the increasing enumeration of $C^j_\alpha$.

We shall define the sequence $t^j_\alpha:\alpha\rightarrow 2$, by recursion on $i<\otp(C^j_\alpha)$,
as the limit of an increasing chain $\{t^j_{\alpha,i}:\alpha^j_i\rightarrow 2\mid i<\otp(C^j_\alpha)\}$, as follows:
\begin{itemize}
\item
For $i=0$, let $t^j_{\alpha,0}$
be the $\preceq$-least element of $T_{\alpha^j_0}$;

\item Suppose that $t^j_{\alpha,i}:\alpha^j_i\rightarrow 2$ has been defined for some $i<\otp(C^j_\alpha)$,
and let us define $t^j_{\alpha,i+1}:\alpha^j_{i+1}\rightarrow 2$.

 Put $\langle k,\nu\rangle:=f(i)$, and $s:=\psi_{\alpha^j_k}(\nu)$. Consider the following set:
 $$P^j_i:=\left\{ p\in \T\restriction\alpha^j_{i+1}\mid s* t^j_{\alpha,i}\s p,\  \varphi(p)\in S_{\alpha^j_{i+1}}\right\}.$$
If $P^j_i\not=\emptyset$, then let $t^j_{\alpha,i+1}$ be the $\preceq$-least element of $T_{\alpha^{j}_{i+1}}$
such that $s*t^j_{\alpha,i+1}$ is above some $p\in P^j_i$. Otherwise, let $t^j_{\alpha,i+1}$ be the $\preceq$-least element of $T_{\alpha^{j}_{i+1}}$
which is above $t^j_{\alpha,i}$.

\item Suppose that $i<\otp(C^j_\alpha)$ is a limit ordinal, and that $t^j_{\alpha,k}:\alpha^j_k\rightarrow 2$ has been defined for all $k<i$.
Then let $t^j_{\alpha,i}:=\bigcup_{k<i}t^j_{\alpha,k}$.

\end{itemize}
This completes the construction of $t^j_\alpha:=\bigcup\{t^j_{\alpha,i}\mid i<\otp(C^j_\alpha)\}$.
Once that $ t^j_\alpha$ is defined for each $j<\lambda$, we let $$T_\alpha:=\{ s*t^j_\alpha\mid s\in(\T\restriction\alpha), j<\lambda\}.$$

In either case, $|T_\alpha|\le\lambda$. Thus, fix an arbitrary surjection $\psi_\alpha:\lambda\rightarrow T\restriction\alpha+1$.

This completes the construction of the tree.

\begin{claim}\label{441} For every ordinal $\alpha<\lambda^+$:
\begin{itemize}
\item if $t\in T_\alpha$ and $s\in\T\restriction\alpha$, then $s*t\in T_\alpha$;
\item $t_\alpha^j\restriction\beta\in T_\beta$ for every $j<\lambda$ and $\beta<\alpha=\sup(\alpha)$;
\end{itemize}
In particular, $T_\alpha$ is well-defined for every $\alpha<\lambda^+$.
\end{claim}
\begin{proof} Suppose not, and let $\alpha<\lambda^+$ be the minimal counter-example to the failure of at least one of the two items.
Clearly, $\alpha$ is a nonzero limit  ordinal.

If $t_\alpha^j\restriction\beta\in T_\beta$ for all $j<\lambda$, and $\beta<\alpha$,
then the minimality of $\alpha$ and the definition of $T_\alpha$ insure that $s*t\in T_\alpha$ for all $t\in T_\alpha$ and $s\in\T\restriction\alpha$.
Thus, fix some $j<\lambda$ such that $t_\alpha^j\restriction\beta\not\in T_\beta$ for some $\beta$.
Let $i<\otp(C^j_\alpha)$ be the least such that $t^j_\alpha\restriction\alpha^j_i\not\in T_{\alpha^j_i}$.
Then $i$ is a limit ordinal, and it must be the case that $T_\alpha$ and $T_{\alpha^j_i}$
have been defined according to case $\beth$.
In this case, $t^j_\alpha\restriction\alpha^j_i$ is completely determined
by $(\T\restriction \alpha_i^j)$, $(C^j_\alpha\cap \alpha_i^j)$, and $\langle S_\gamma\mid \gamma<\alpha_i^j\rangle$.
Let $j'$ be such that $C^j_\alpha\cap\alpha_i^j=C^{j'}_{\alpha_i^j}$.
Then the construction of $t^{j'}_{\alpha_i^j}$ is determined by the above-mentioned three objects,
in a way that implies $t^j_\alpha\restriction\alpha^j_i=t^{j'}_{\alpha_i^j}$.

It follows that $t^{j'}_{\alpha_i^j}\not\in T_{\alpha^j_i}$, contradicting the fact that $\alpha_i^j<\alpha$,
while $\alpha$ was chosen as the minimal level at which a counter-example exists.
\end{proof}

So $\T$ is a tree of height $\lambda^+$, and width $\lambda$. We now continue with its analysis.
\begin{claim} $\T$ is $\kappa$-complete.
\end{claim}
\begin{proof} Suppose that $\langle s_i\mid i<\theta\rangle$ is a strictly increasing sequence of elements of $T$,
with $\theta<\kappa$. Put $s:=\bigcup_{i<\theta}s_i$, and $\alpha:=\dom(s)$. Then $\cf(\alpha)<\kappa\le\cf(\lambda)$.
So $\otp(C^j_\alpha)<\lambda$ for all $j<\lambda$, and $|T_\alpha|^{\cf(\alpha)}\le\lambda^{\theta}=\lambda$,
which implies that $T_\alpha$ has been defined according to case $\aleph$. In particular, $s\in T_\alpha$.
\end{proof}

For $s\in\T$, denote $\T^s:=\{ t\in{}^{<\lambda^+}2\mid s^\frown t\in \T\}$.
\begin{claim}   $\T^{s_0}=\T^{s_1}$ for all $s_0,s_1\in\T$ of the same height.

In particular, $\T$ is homogenous.
\end{claim}
\begin{proof} This is an immediate consequence of Claim \ref{441}.
\end{proof}

Thus, we are left with establishing the following.

\begin{claim} $\T$ is a $\lambda^+$-Souslin tree.
\end{claim}
\begin{proof} Towards a contradiction, suppose that $A\s{}^{<\lambda^+}2$ is a maximal antichain of size $\lambda^+$.
Put
\begin{itemize}
\item $D:=\{ \gamma<\lambda^+\mid \T\cap \varphi^{-1}[\gamma]=\T\restriction\gamma\}$;
\item $E:=\{\gamma<\lambda^+\mid A\cap\gamma\text{ is a maximal antichain in }\T\restriction\gamma\}$.
\end{itemize}
Then $D\cap E$ is a club.
Put $$S:=\{\gamma\in D\cap E\mid \varphi[A]\cap\gamma=S_\gamma\}.$$
Then $S$ is a stationary set, and we may fix some $\alpha<\lambda^+$
such that $\otp(C^j_\alpha)=\lambda$ and $\nacc(C^j_\alpha)\s S$ for all $j<\lambda$.

Since $|A|=\lambda^+$, let $q$ be some element of $A$ with $\dom(q)>\alpha$.
Consider $q\restriction\alpha$. Then $q\restriction\alpha\in T_\alpha$,
and hence $q\restriction\alpha=s*t^j_\alpha$ for some $s\in\T\restriction \alpha$ and $j<\lambda$,
which we now fix.
Let $\{ \alpha^j_i\mid i<\lambda\}$ denote the increasing enumeration of $C^j_\alpha$,
and let $k$ be large enough so that $s\in T\restriction \alpha^j_k$. Fix $\nu<\lambda$
such that $\psi_{\alpha^j_k}(\nu)=s$, and $i<\lambda$ such that $f(i)=(k,\nu)$.
Since $\alpha^j_{i+1}\in\nacc(C^j_\alpha)\s S$, and $s*t^j_{\alpha,i}\in T\restriction\alpha^j_{i+1}$,
there exists some
$p\in \T\restriction\alpha^j_{i+1}$ with $\varphi(p)\in S_\gamma$
such that either $p\s s*t^j_{\alpha,i}$,
or $s*t^j_{\alpha,i}\s p$. In either case, we get that $s* t^j_{\alpha,i+1}$ is above some element of $A$.
In other words, for some $p\in A$, we have:
$$p\s s*t^j_{\alpha,i+1}\s s*t^j_\alpha=q\restriction\alpha\subset q,$$
contradicting the fact that $p,q$ are distinct elements of the antichain $A$.
\end{proof}
\end{proof}

To conclude this section, we point out that $\sc_{\lambda,\lambda}^{\{\lambda\},\lambda}$ and $\sc_{\lambda,\lambda}^{\{\lambda\},1}$
are equivalent for every cardinal $\lambda$, and that
$\sc_{\lambda,\lambda}^{\{\lambda\},1}$ and $\clubsuit(E^{\lambda^+}_\lambda)$ are equivalent,
provided that $\lambda^{<\lambda}=\lambda$.
As a corollary, we derive the following folklore fact concerning successors of regulars.
\begin{cor} If $\lambda^{<\lambda}=\lambda$ is an uncountable regular cardinal and $\diamondsuit({E^{\lambda^+}_\lambda})$ holds,
then there exists an homogenous $\lambda$-complete, $\lambda^+$-Souslin tree.
\end{cor}

\npg\section{Connecting the dots}\label{sec5}
\begin{THMA} Suppose that $\ch_\lambda$ holds for a given uncountable cardinal $\lambda$.

Then all of the following are equivalent:
\begin{itemize}
\item $\square_\lambda$;
\item $\sc_\lambda(S)$, for every stationary  $S\s E^{\lambda^+}_{\not=\cf(\lambda)}$;
\item $\sc_\lambda(S)$, for every $S\s \lambda^+$ that reflects stationarily often.
\end{itemize}
\end{THMA}
\begin{proof}It is obvious that either of the $\sc_\lambda$ principles implies
$\square_\lambda$, so let us focus on the other implications.
Suppose that $\square_\lambda+\ch_\lambda$ holds.

$\blacktriangleright$ By Theorem \ref{0034},
we get that $\sd'_\lambda(S)$ holds for every stationary  $S\s E^{\lambda^+}_{>\omega}\cap E^{\lambda^+}_{\not=\cf(\lambda)}$,
and every $S\s \lambda^+$ that reflects stationarily often.
It now follows from Lemma \ref{18}, that
$\sc_\lambda(S)$ holds for every stationary  $S\s E^{\lambda^+}_{>\omega}\cap E^{\lambda^+}_{\not=\cf(\lambda)}$,
and every $S\s \lambda^+$ that reflects stationarily often.

$\blacktriangleright$ Suppose that $S\s E^{\lambda^+}_\omega$ is a given stationary set, while $\cf(\lambda)>\omega$.
Pick a $\square_\lambda$-sequence $\langle C_\alpha\mid\alpha<\lambda^+\rangle$.
By Fodor's lemma, pick a limit $\theta<\lambda$ such that $\{ \alpha\in S\mid \otp(C_\alpha)=\theta\}$ is stationary,
and denote the latter by $S'$.
By $S'\s E^{\lambda^+}_{\not=\cf(\lambda)}$,
$\ch_\lambda$ and \cite{sh922}, we get that $\diamondsuit({S'})$ holds, hence, it easy to find for every $\alpha\in S'$,
a set $A_\alpha$ of order-type $\omega$, for which $\{ \alpha\in S'\mid A_\alpha\s A\}$
is stationary for every cofinal $A\s\lambda^+$.

Finally, for all limit $\alpha<\lambda^+$, define:
$$C_\alpha':=\begin{cases}A_\alpha,&\alpha\in S'\\
C_\alpha,&\alpha\not\in S', \otp(C_\alpha)\le\theta\\
\{\beta\in C_\alpha\mid \otp(C_\alpha\cap\beta)>\theta\},&\alpha\not\in S', \otp(C_\alpha)>\theta
\end{cases}.$$
Then $\langle C_\alpha'\mid \alpha<\lambda^+\rangle$ witnesses $\sc_\lambda(S)$.
\end{proof}

\begin{THMB}  $\square_\lambda$ implies $\sc_\lambda$, provided that:
\begin{itemize}
\item $\lambda$ is a limit uncountable cardinal, and $\lambda^\lambda=\lambda^+$;
\item $\lambda$ is a successor cardinal, and $\lambda^{<\lambda}<\lambda^\lambda=\lambda^+$.
\end{itemize}
\end{THMB}
\begin{proof}
If $\lambda$ is a successor cardinal, then simply appeal to Theorem \ref{301}.
Now, suppose that $\lambda$ is a limit uncountable cardinal.
By Theorem \ref{61}, we get that $\sd_\lambda^{\Gamma}$ holds for some cofinal subset $\Gamma\s\reg(\lambda)$.
Since $\lambda$ is a limit cardinal, this means that $\sup(\Gamma)=\lambda$.
Then, by Lemma \ref{25}, $\sc_\lambda$ is valid.
\end{proof}

\begin{THMC} Suppose that $\square_\lambda$ holds for a given singular cardinal $\lambda$.

If $\lambda^{<\cf(\lambda)}<\lambda^\lambda=\lambda^+$,
then there exists an homogenous $\lambda^+$-Souslin tree,
which is moreover $\cf(\lambda)$-complete.
\end{THMC}
\begin{proof} Since $\lambda$ is a limit cardinal, we get from Theorem B, that $\sc_\lambda$ holds. It now follows from $\lambda^{<\cf(\lambda)}=\lambda$
and Lemma \ref{lemma42}, that $\sc_{\lambda,\lambda}^{\{\lambda\},1}$ is valid.
So $\cf(\lambda)=\min\{\theta\mid \lambda^\theta\not=\lambda\}$ and $\ch_\lambda+\sc_{\lambda,\lambda}^{\{\lambda\},\lambda}$ holds,
meaning that we are in a position to invoke Proposition \ref{43}.
\end{proof}

\begin{THMD} For an uncountable cardinal $\lambda$, the following are equivalent:
\begin{itemize}
\item $\square_\lambda+\ch_\lambda$;
\item $\sd^{\Gamma}_{\lambda}$, for $\Gamma=\reg(\lambda)$;
\item $\sd_\lambda(S)$, for every stationary  $S\s E^{\lambda^+}_{>\omega}\cap E^{\lambda^+}_{\not=\cf(\lambda)}$;
\item $\sd_\lambda(S)$, for every $S\s \lambda^+$ that reflects stationarily often.
\end{itemize}
\end{THMD}
\begin{proof}It is obvious that either of the $\sd_\lambda$ principles implies
$\square_\lambda$ and $\ch_\lambda$, thus let us focus on the other implications. Suppose that $\square_\lambda+\ch_\lambda$ holds.

$\blacktriangleright$ By Theorem \ref{0034} and Lemma \ref{18},
we get that $\sd_\lambda(S)$ holds for every stationary  $S\s E^{\lambda^+}_{>\omega}\cap E^{\lambda^+}_{\not=\cf(\lambda)}$,
and every $S\s \lambda^+$ that reflects stationarily often.

$\blacktriangleright$ If $\lambda$ is the successor of some regular cardinal $\kappa$, then $E^{\lambda^+}_\kappa$ reflects
 stationarily often and hence $\sd_\lambda(E^{\lambda^+}_\kappa)$ holds.
If $\lambda$ is a limit uncountable cardinal, then by Theorem \ref{61}, we get that $\sd^{\Gamma}_{\lambda}$ holds for some cofinal $\Gamma\s\reg(\lambda)$.
Thus, in any case, we may find some $\Gamma\s\lambda^+$ with $\Gamma\bks\theta\not=\emptyset$ for all $\theta<\lambda$,
for which $\sd_\lambda^\Gamma$ holds.
Pick a witness $\langle (C_\alpha,S_\alpha)\mid \alpha<\lambda^+\rangle$,
and let us show that the very same sequence witnesses $\sd_\lambda^{\reg(\lambda)}$.
Suppose that $D\s\lambda^+$ is a club,  $A$ is some cofinal subset of $\lambda^+$,
and $\theta\in\reg(\lambda)$. As $\Gamma\bks\theta\not=\emptyset$,
we may pick some $\alpha\in E^{\lambda^+}_{\ge\theta}$ for which $C_\alpha\s D$, $S_\alpha=A\cap\alpha$ and $\sup(\acc(C_\alpha))=\alpha$.
By $\cf(\sup(\acc(C_\alpha)))\ge\theta$, there exists some $\beta\in\acc(C_\alpha)\cup\{\alpha\}$ with $\cf(\acc(C_\alpha)\cap\beta)=\theta$.
For such a $\beta$, we have $C_\beta\s C_\alpha\s D$, $S_\beta=S_\alpha\cap\beta=A_\alpha\cap\beta$, and $\sup(\acc(C_\beta))=\beta$.
\end{proof}

\npg\section{Remarks and Open Problems}\label{sec6}
\subsection*{Remarks}

The technology of this paper also yields the following results.

\begin{prop} If $\square_\lambda+\ch_\lambda$ holds for a given singular cardinal,
then there exists a sequence $\langle S_\alpha\mid \alpha<2^{\lambda^+}\rangle$ such that:
\begin{enumerate}
\item $|S_\alpha\cap S_\beta|<\lambda^+$ whenever $\alpha<\beta<2^{\lambda^+}$;
\item for all $\alpha<2^{\lambda^+}$, there exists some $(<\lambda^+)$-distributive forcing extension
in which $S_\alpha$ contains a club subset of $\lambda^+$.
\end{enumerate}
\end{prop}
\begin{note}
The existence of the above sort of sequence for $\lambda=\omega$ is a well-known consequence of $\diamondsuit(\omega_1)$.
\end{note}
\begin{prop} If $\square_\lambda$ holds for a given uncountable cardinal $\lambda$,
then this may be witnessed by a coherent sequence $\langle C_\alpha\mid\alpha<\lambda^+\rangle$
with the following additional feature:
\begin{itemize}
\item[(3)] for every club $D\s\lambda^+$ and every limit $\theta<\lambda$,
there exists some $\alpha<\lambda^+$ such that $C_\alpha\s D$, and $\otp(C_\alpha)=\theta$.
\end{itemize}
\end{prop}
\begin{note} So $\alpha\mapsto\otp(C_\alpha)$ yields a canonical partition of $\acc(\lambda^+)$
into $\lambda$-many mutually disjoint stationary sets.
\end{note}

\begin{prop} If $\kappa$ is a regular infinite cardinal, and $S\s E^{\kappa^{++}}_\kappa$ is stationary,
then the following are equivalent:
\begin{enumerate}
\item $2^{\kappa^+}=\kappa^{++}$;
\item there exists a sequence $\langle (C_\alpha,S_\alpha)\mid \alpha<\kappa^{++}\rangle$
such that for every club $D\s\kappa^{++}$ and every $A\s\kappa^{++}$,
there exists some $\delta\in S$ with:
\begin{enumerate}
\item $C_\delta$ is a club in $\delta$;
\item $C_\delta\cup\{\delta\}\s \{\alpha\in D\mid A\cap\alpha=S_\alpha\}$.
\end{enumerate}
\end{enumerate}
\end{prop}
\begin{note}
The above may be seen as a unified evidence to the fact \cite{sh922} that $2^{\kappa^+}=\kappa^{++}$
entails $\diamondsuit(E^{\kappa^{++}}_\mu)$ for all regular $\mu\le\kappa$.
\end{note}

\begin{remark} It follows from Theorem B, that if $W$ is an inner model of $\zfc+\gch+\forall\lambda\square_\lambda$,
and the covering lemma holds for $W$ (e.g., if $W=L$ and $0^\sharp$ does not exist),
then $\sc_\lambda$ and $\sc_{\lambda,\lambda}^{\{\lambda\},1}$ are valid for every singular strong limit cardinal $\lambda$.
\end{remark}

\begin{remark} For the reader who is interested in the dual concept of homogeneity,
we mention that obtaining a \emph{rigid} Souslin tree is considerably easier. For instance,
in \cite{rinot12}, for every uncountable cardinal $\lambda$, a rigid $\lambda^+$-Souslin tree is constructed merely from $\gch+\neg\refl(E^{\lambda^+}_{\not=\cf(\lambda)})+\square^*_\lambda$.
\end{remark}
\subsection*{Open Problems}

\begin{q} By the results of this paper, $\gch+\square_\lambda$ entails $\sc_{\lambda,\lambda}^{\{\lambda\},1}$
for every singular cardinal $\lambda$. Could this be improved to get $\sc_{\lambda,1}^{\{\lambda\},1}$?
\end{q}

For a regular cardinal $\lambda$, we already know that $\textsf{GCH}+\square_\lambda\not\Rightarrow \sc_{\lambda,1}^{\{\lambda\},1}$.

\begin{q} By Theorem B, $\sc_\lambda$ and $\square_\lambda$ are equivalent,
assuming some fragments of $\gch$. Are these  fragments of $\gch$ necessary?
\end{q}

By Proposition \ref{p34}, $\sc_\lambda$ is indeed
equivalent to $\square_\lambda+\ch_\lambda$  whenever $\lambda$ is a singular cardinal,
hence, the open problems are:

(1) Whether $\sc_\lambda$ implies $\lambda^\lambda=\lambda^+$ for every regular cardinal $\lambda$;

(2) Whether $\sc_\lambda$ implies $\lambda^{<\lambda}=\lambda$ for every successor cardinal $\lambda$.

Finally, it is still conceivable that the weaker principle, $\sc_\lambda(S)$, has no effect on cardinal arithmetic.
So, let us ask about the simplest case:
\begin{q} Does $\square_\lambda$ imply $\sc_\lambda(E^{\lambda^+}_\omega)$ for every uncountable cardinal $\lambda$?
what about $\lambda$ which is a singular strong limits of uncountable cofinality?
\end{q}

\newpage
\bibliographystyle{plain}

\end{document}